\documentclass[11pt,oneside,english]{amsart}
\usepackage[T1]{fontenc}
\usepackage[latin1]{inputenc}
\usepackage{verbatim}
\usepackage{mathrsfs}
\usepackage{amstext}
\usepackage{amsthm}
\usepackage{amssymb}
\usepackage{stmaryrd}
\usepackage{stackrel}

\usepackage{stackengine}%para escribir un texto sobre otro
\usepackage{tikz} %para hacer diagramas complicados
\usetikzlibrary{shapes.geometric, arrows}
\usetikzlibrary{er,positioning}
\usepackage{tikz-cd}

\usepackage{anysize}
\usepackage{color}

\makeatletter

\numberwithin{equation}{section}
\numberwithin{figure}{section}
\theoremstyle{plain}
\newtheorem{thm}{\protect\theoremname}[section]
  \theoremstyle{plain}
  \newtheorem{prop}[thm]{\protect\propositionname}
  \theoremstyle{definition}
  \newtheorem{defn}[thm]{\protect\definitionname}
  \theoremstyle{plain}
  \newtheorem{cor}[thm]{\protect\corollaryname}
  \theoremstyle{plain}
  \newtheorem{lem}[thm]{\protect\lemmaname}
  \theoremstyle{remark}
  \newtheorem{rem}[thm]{\protect\remarkname}
  \theoremstyle{definition}
  
  \theoremstyle{plain}
  \newtheorem*{thm*}{\protect\theoremname}

\makeatother

\usepackage{babel}
  \providecommand{\corollaryname}{Corollary}
  \providecommand{\definitionname}{Definition}
  \providecommand{\examplename}{Example}
  \providecommand{\lemmaname}{Lemma}
  \providecommand{\propositionname}{Proposition}
  \providecommand{\remarkname}{Remark}
\providecommand{\theoremname}{Theorem}

\marginsize{3.5cm}{3.3cm}{3.7cm}{3cm}

\begin{document}
\title{A General Theory of Tensor Products of Convex Sets in Euclidean Spaces}

\author{Maite Fern{\'a}ndez-Unzueta$^1$\and Luisa F. Higueras-Monta{\~n}o$^2$}
\address{Centro de Investigaci\'on en Matem\'aticas (Cimat), A.P. 402 Guanajuato, Gto., M\'exico}
\email{maite@cimat.mx$^1$, fher@cimat.mx$^2$ }
\begin{abstract}

We introduce both the notions of tensor product of  convex bodies that  contain  zero in the  interior, and of
 tensor product of $0$-symmetric convex bodies in Euclidean spaces. We prove  that there  is a bijection  between tensor products of $0$-symmetric convex bodies and tensor norms on finite dimensional spaces. This bijection preserves duality, injectivity and projectivity.
  We obtain a  formulation of Grothendieck`s Theorem for $0$-symmetric convex bodies and use it to give a geometric representation (up to the $K_G$-constant) of the   Hilbertian tensor product.
 We see that  the property of  having  enough symmetries is preserved by these tensor products, and
 exhibit  relations    with the    L{\"o}wner and the John ellipsoids.

 \

 \noindent{{\it Keywords:} Convex body, Tensor product of convex sets,  Tensor product of Banach spaces,  Hilbertian tensor norm, Ideals of linear operators, Grothendieck's inequality. }
\newline

\noindent{{\it 2000 Mathematics Subject Classification:} 46M05, 52A21, 47L20, 15A69.}
\end{abstract}

\maketitle

\section{Introduction}\label{sec: intr}

Given two compact convex sets $K_1$ and $K_2$, there  are several different manners to  construct  a new compact convex set that  can be considered as a   tensorial product of them. Some of these notions  were studied in relation with the so-called Choquet Theory, as done by Z. Semadeni in \cite{Semadeni1965}, I. Namioka and R.R. Phelps in \cite{Namioka1969}
 or E. Behrends and G. Wittstock in \cite{BehrendsWittstock}.
 More recently, tensor  products of convex sets  have been studied in relation with the so called  quantum information theory, as can be found in \cite{Aubrun2006}
 by G. Aubrun and S. Szarek. Other classes of tensor products as well as some properties of those already mentioned can be found in \cite{Davies1969,hulanickiphelps,lazar,Velasco2015}.

 In a previous paper, we characterized when a centrally symmetric convex body $B$ in the Euclidean space $\mathbb{R}^d$ of dimension $d=d_1\cdots d_l$, is the unit ball of a Banach space whose norm is a reasonable crossnorm (see \cite[Theorem 3.2]{tensorialbodies}). That is,  we showed how  to determine in purely geometric terms, if $B$ is the unit ball of a reasonable crossnorm on a tensor product space $\otimes_{i=1}^{l}\left(\mathbb{R}^{d_i},\Vert\cdot\Vert_i\right),$ when the norms on each factor are not determined \textit{a priori}.  In this same line,  we develop here a theory of tensor products of centrally symmetric convex bodies, which is
consistent with the theory of tensor norms on  finite dimensional Banach spaces. Our main result is that there  is a bijection  between both, tensor norms on finite dimensions and tensor products of $0$-symmetric convex bodies. Moreover,  this bijection preserves duality, injectivity and projectivity (Theorem \ref{thm: main theorem the bijection}).

The theory of tensor norms was mainly developed by A. Grothendieck  \cite{Grothendieck1956}. It  is a fundamental tool in the modern study of Banach spaces. Among other things, it establishes the foundations of the theory of ideals of linear operators and the local theory of Banach spaces, as can bee seen in \cite{defantfloret,DiestelJarchTong,diestelfourie,Ryan2013,Tomczak-Jaegermann1989}. Its influence extends to a wide range of areas from Mathematical Analysis to Graph Theory or Computer Science \cite{efremenko,KhotNaor,pisier}. When  the  spaces are finite dimensional,   Theorem \ref{thm: main theorem the bijection} translates  this tool  into  the context of   Convex Geometry.  Indeed,  a geometric formulation of   Grothendieck`s Theorem is provided in Section 4.

We now briefly expose the contents of the paper. We begin by developing   a theory of tensor products of convex bodies with $0$ in the interior.  In this part (Section \ref{sec: tensorial convex bodies}) we do not assume that the sets are centrally symmetric. In this study, two particular tensor products play a fundamental role, namely the projective  and the injective tensor products $\otimes_{\pi},\,\otimes_{\epsilon}$. Their main properties, that  will be used frequently afterwards, are stated.

    If $P_{i}\subset\mathbb{E}_{i}$, $\dim\mathbb{E}_i={d_i},$ $i=1,...,l,$ are convex bodies with zero in the interior, we  say that a compact convex set $Q\subset \mathbb{R}^{d_1\cdots d_l}$ is a {\bf tensorial convex body with respect to $P_i$, $i=1,\ldots,l,$} if  $P_{1}\otimes_{\pi}\cdots\otimes_{\pi}P_{l}\subseteq P  \subseteq P_l \otimes_{\epsilon}\cdots\otimes_{\epsilon}P_{l}.$ A \textbf{tensor product of convex bodies} $\otimes_{\alpha}$ (Definition \ref{defn: tensor prod of cb}) assigns  to each $l$-tuple as before, a tensorial convex body  with respect to $P_i$, $P_{1}\otimes_{\alpha}\cdots\otimes_{\alpha}
  P_{l}\subset\otimes_{i=1}^{l}\mathbb{\mathbb{E}}_{i}$ that satisfies  the following uniform property:
   For every linear mapping $T_{i}:\mathbb{E}_{i}\rightarrow\mathbb{F}_{i},$ $i=1,..,l.$ If $P_i\subset\mathbb{E}_i$, $Q_i\subset\mathbb{F}_i,$ $i=1,\ldots,l,$ are convex bodies with $0$ in the interior and $T_i\left(P_i\right)\subseteq Q_i,$ then
$$
T_{1}\otimes\cdots\otimes T_{l}\left(P_{1}\otimes_{\alpha}\cdots\otimes_{\alpha}P_{l}\right)\subseteq Q_{1}\otimes_{\alpha}\cdots\otimes_{\alpha}Q_{l}.
$$
An important fact derived from such uniform property is that tensor products of convex bodies  are invariant under the tensor product of linear isomorphisms (see Proposition \ref{prop:linear invariance of products}). This also implies that, in the case of $0$-symmetric convex bodies, these tensor products are continuous with respect to the Banach-Mazur distance (Proposition \ref{prop:continuity wrt BM distance}).

In  Definiton \ref{defn: dual tensor prod of cb}, we introduce duality of tensor products    by means of   polarity. It holds that the dual $\otimes_{\alpha^\prime} $ of a tensor product of convex bodies $\otimes_{\alpha} $ is a tensor product of convex bodies, as well. It satisfies  $\otimes_{\alpha^{\prime \prime}}=  \otimes_{\alpha}. $

The injective and the projective properties of a tensor product of convex bodies are stated in terms of preservation under taking    sections and quotients, respectively (see Definition \ref{defn: injective and proj tpb}). The fundamental relation between tensor products of convex bodies and their duals is exhibited in Theorem \ref{thm:duality inj and proj}, where we prove that duals of projective tensor products of convex bodies are injective tensor products of convex bodies and vice versa.

%In section 2.4, we show that given a tensor product of convex bodies $\otimes_\alpha$, it is possible to construct injective $\otimes_{\alpha\setminus}$ and projective tensor products $\otimes_{\alpha/}$ associated to it.

  In Section \ref{sec: tensor product 0symm} is where we  apply these results to develop  a theory of tensor products of $0$-symmetric convex bodies. It is in  this context where we state our main result, Theorem \ref{thm: main theorem the bijection},  which provides the existence of a bijection between tensor products of $0$-symmetric convex bodies and tensor norms on finite dimensional spaces.   It  widens  the scope of the multilinear generalization of Minkowski`s bijection between norms and centrally symmetric convex bodies, established in  Theorem 3.2 and Corollary 3.4 of  \cite{tensorialbodies}.

  We study some  important geometric aspects of these tensor products, as  is the case of the Banach-Mazur distance and  the injective and the projective properties. In Subsection \ref{subs: explicit hulls} que provide explicit representations of the biggest injective and the smallest projective tensor products.
 In Subsection \ref{subs:John ellipsoid} we  study the L{\"o}wner and John ellipsoids associated to a tensor product of $0$-symmetric convex bodies. We also  prove that tensor products of convex bodies with enough symmetries have enough symmetries, too (Proposition \ref{prop:convex bodies with enough symmetries}).

  The results presented so far are valid in the case of tensor products of any fixed number of spaces $l\in\mathbb{N}$. In the case of products of two factors, there is also a bijection with the ideals of linear operators  preserving duality, injectivity and projectivity (see Corollary \ref{cor: ideals tensor norms tensor convex}).  This is studied in Section \ref{sect: Hilbertian tsr pdct},  where we develop in more detail the  Hilbertian tensor product $\omega_2$,  and give the geometric formulation of Grothendieck`s Theorem, see (\ref{eq: groth ineq elinfty}).

%The names ``injective'' and ``projective'' as well as the notation used in this section obeys the fact that when $0$-symmetric convex bodies are considered, injective tensor products of convex bodies (resp. projective) corresponds, under the bijection given in Theorem \ref{thm: main theorem the bijection}, to injective tensor norms (resp. projective). In the same way, $\otimes_{\alpha\setminus}$ and $\otimes_{\alpha/}$ correspond to the injective and projective associate tensor norms. %respectively (Propositions \ref{prop: dual norms and dual products} and \ref{prop:correspondence inje norm inje prods})
%Propositions \ref{prop: dual norms and dual products} and \ref{prop:correspondence inje norm inje prods}.

\subsection{Notation and preliminaires}

Throughout this paper the letters $\mathbb{E},\mathbb{F}$ or $\mathbb{E}_i,\mathbb{F}_i,$ will denote Euclidean spaces over $\mathbb{R}$. The scalar product on $\mathbb{E}$ will be denoted by $\left\langle \cdot,\cdot\right\rangle_{\mathbb{E}}$ and the associated Euclidean ball by $B_{\mathbb{E}}.$ Given a linear map $T:\mathbb{E}\rightarrow\mathbb{F},$ its transpose    $T^{t}$ is  the linear map such that $\left\langle Tx,y\right\rangle_{\mathbb{F}}=\left\langle x,T^{t}y\right\rangle_{\mathbb{E}}.$

 Given vector spaces  $V_{1},\cdots,V_l, W$ over the same field ($\mathbb{R}$ or $\mathbb{C}$)  and a multilinear mapping $T:V_{1}\times\cdots\times V_{l}\rightarrow W$, we will denote $\hat{T}:\otimes_{i=1}^{l}V_{i}\rightarrow W$ the unique linear mapping such that
  $T=\hat{T}\circ \otimes$, where $\otimes$ denotes the canonical multilinear mapping:
\begin{alignat*}{1}
\otimes:V_{1}\times\cdots\times V_{l} & \longrightarrow\otimes_{i=1}^{l}V_{i}\\
\left(x^{1},\ldots,x^{l}\right) & \rightarrow x^{1}\otimes\cdots\otimes x^{l}.
\end{alignat*}

When the spaces  $\mathbb{E}_i$,  $i=1,\ldots,l$, are  Euclidean spaces,  $\otimes_{i=1}^l\mathbb{E}_i$ has a natural scalar product: it is defined in decomposable vectors as
\[
\left\langle x^{1}\otimes\cdots\otimes x^{l},y^{1}\otimes\cdots\otimes y^{l}\right\rangle _{H}:={\Pi}_{i=1}^l\left\langle x^i,y^i\right\rangle_{\mathbb{E}_i},
\]
and extended to $\otimes_{i=1}^l\mathbb{E}_i$ by multilinearity. The space $\otimes_{H,i=1}^l\mathbb{E}_i:=\left(\otimes_{i=1}^l\mathbb{E}_i,\left\langle \cdot,\cdot \right\rangle _{H}\right)$ is called \textbf{the Hilbert tensor product of $\mathbb{E}_1,\ldots,\mathbb{E}_l$} (the details can be seen in \cite[Section 2.6]{kadisonkingrose}).
Unless otherwise is stated, we will always assume that this is the euclidean structure defined in the tensor product of euclidean spaces.

Every compact convex set  $P\subset\mathbb{E}$ with nonempty interior, $\text{int}(P)\neq\emptyset$, is called a convex body. In addition, if $P=-P$ then $P$ is called a $0$-symmetric convex body. We write $\mathcal{B}\left(\mathbb{E}\right)$ to denote the set of $0$-symmetric convex bodies contained in $\mathbb{E}.$

Given a nonempty set $C\subset\mathbb{E}$, its polar set is $C^{\circ}:=\left\{ y\in\mathbb{E}:{\sup}_{x\in C}\left\langle x,y\right\rangle _{\mathbb{E}}\leq1\right\}.$
The \textbf{Minkowski functional} (or gauge function) of $P\in\mathcal{B}\left(\mathbb{E}\right)$  is defined as
\[
g_{P}\left(x\right):=\inf\left\{ \lambda>0:\lambda^{-1}x\in P\right\} \text{ for }x\in\mathbb{E}.
\]
A fundamental result concerning  $0$-symmetric convex bodies is the bijection between norms defined on $\mathbb{E}$ and $0$-symmetric convex bodies, given by the Minkowski functional. This result, originally  due to H. Minkowksi \cite{Minkowski1927}, establishes that the map
\begin{align*}
\mathcal{B}\left(\mathbb{E}\right) & \longrightarrow\left\{\text{norms defined on }\mathbb{E}\right\}\\
P & \longrightarrow g_P\left(\cdot\right)
\end{align*}
is a bijection, the unit ball of the space $\left(\mathbb{E},g_P\right)$ is $P$ and $g_{P^{\circ}}\left(x\right)=\left\Vert \left\langle \cdot,x\right\rangle :\left(\mathbb{E},g_p\right)\rightarrow\mathbb{R}\right\Vert.$

The proof of the previous result as well as  the  fundamentals about convex bodies and Convex Geometry that will be used in this paper can be found  in  \cite{Schneider1993}.

Recall that whenever  $X_1,\ldots,X_l$ are Banach spaces,  the projective tensor norm $\pi$ and the injective tensor norm $\epsilon$ on
$\otimes_{i=1}^l{X}_i$ are defined as:
\[
\pi\left(u\right):=\inf\left\{ {\sum}_{j=1}^{n}\| x_{j}^{1}\| \cdots\| x_{j}^{l}\| :u={\sum}_{j=1}^{n}x_{j}^{1}\otimes\cdots\otimes x_{j}^{l}\right\}
\]
and
\[
\epsilon\left(u\right):=\sup\left\{ \left|x_{1}^{*}\otimes\cdots\otimes x_{l}^{*}\left(u\right)\right|:x_{i}^{*}\in B_{X_{i}^{*}},\textrm{ for }i=1,\ldots,l\right\}.
\]
for $u\in\otimes_{i=1}^{l}X_i.$

A norm $\alpha\left(\cdot\right)$ on the tensor product $\otimes_{i=1}^{l}X_{i}$ is a \textbf{reasonable crossnorm} if
%\begin{enumerate}
%\item $\alpha\left(x^{1}\otimes\cdots\otimes x^{l}\right)\leq\left\Vert x^{1}\right\Vert \cdots\left\Vert x^{l}\right\Vert $
%for every $x^{i}\in X_{i}$ with $i=1,...,l.$
%\item If $x_{i}^{*}\in X_{i}^{*},$ $i=1,...,l,$ then $x_{1}^{*}\otimes\cdots\otimes x_{l}^{*}\in\left(\otimes_{i=1}^{l}X_{i},\alpha\right)^{*}$
%and $\left\Vert x_{1}^{*}\otimes\cdots\otimes x_{l}^{*}\right\Vert \leq\left\Vert x_{1}^{*}\right\Vert \cdots\left\Vert x_{l}^{*}\right\Vert .$
%\end{enumerate}
%Indeed, An useful description of reasonable crossnorms is the following: a norm $\alpha(\cdot)$ is a reasonable crossnorm if
\begin{equation}
\label{eq:reasonablecrossnorms def}
\epsilon\left(u\right)\leq\alpha\left(u\right)\leq\pi\left(u\right)\text{ for every }u\in\otimes_{i=1}^{l}X_{i}.
\end{equation}

Both the projective and the injective tensor norms are reasonable crossnorms. Clearly, they are, respectively, the biggest and the smallest reasonable crossnorm. We suggest the monographs \cite{defantfloret,Ryan2013} for a thorough treatment of tensor products of Banach spaces as well as for the fundamental properties of the projective and the injective tensor norms.

\section{Tensorial convex bodies and tensor products of convex bodies}\label{sec: tensorial convex bodies}

In the theory of tensor norms there is a significant  distinction between the notions of  {\sl reasonable crossnorm}  and of  {\sl tensor norm}. The first one  refers to a norm on the tensor product of some fixed Banach spaces for which (\ref{eq:reasonablecrossnorms def}) holds. The second  refers to a norm that,  in addition, can be defined  on the tensor product of any Banach spaces and such that it preserves the continuity of some mappings (the so called uniform property, see e.g. \cite[pp.10]{diestelfourie}).  This subtle distinction is also present when considering  products of convex sets (compare Definitions  \ref{def:tensorial convex bodies} and \ref{defn: tensor prod of cb} below).

In
 this section the convex sets will be
 fixed and they are not assumed to be  $0$-symmetric.

\subsection{The projective and the injective tensor product of convex bodies}\label{sec:inj proj tensor}

Given  $l$ convex bodies containing $0$ in the interior,
$P_{i}\subset\mathbb{E}_{i},$ $i=1,\ldots,l,$ their
\textbf{projective tensor product} is defined as  the following convex body in $\otimes_{i=1}^l\mathbb{E}_i$:
\begin{equation}\label{eq: proj tensor}
P_{1}\otimes_{\pi}\cdots\otimes_{\pi}P_{l}:=\text{conv}\left\{x^{1}\otimes\cdots\otimes x^{l}\in\otimes_{i=1}^{l}\mathbb{E}_{i}:x^{i}\in P_{i}, i=1,\ldots,l\right\}.
\end{equation}
In this case, $0$ is also in the interior of $P_{1}\otimes_{\pi}\cdots\otimes_{\pi}P_{l}.$
  Their\textbf{ injective tensor product} is the convex body in $\otimes_{i=1}^l\mathbb{E}_i$ defined as:
\begin{equation}\label{eq: inj tensor}
P_{1}\otimes_{\epsilon}\cdots\otimes_{\epsilon}P_{l}:
=\left(P^{\circ}_{1}\otimes_{\pi}\cdots\otimes_{\pi}P^{\circ}_{l}
\right)^{\circ}.
\end{equation}

Here, the polarity is the one determined by $\left\langle \cdot,\cdot \right\rangle _{H}.$ It holds that $0\in\text{int}(P_{1}\otimes_{\epsilon}\cdots\otimes_{\epsilon}P_{l})$, since  such  property  is preserved by taking polars   \cite[Theorem 1.6.1.]{Schneider1993}.

  These products will play a fundamental role in what follows. They  were introduced in \cite{Aubrun2006} and  \cite[Section 4.1]{Aubrun2017}, respectively.

 The projective tensor product satisfies a universal property, that we now describe. Let us consider  $l$ convex bodies with $0$ in the interior, $P_{i}\subset\mathbb{E}_{i}$ $i=1,\ldots,l.$  We say that the pair $(P, \varphi)$ of
a convex body $P\subset\mathbb{E},$ $0\in\text{int}(P),$ and a multilinear mapping $\varphi:\mathbb{E}_{1}\times\cdots\times\mathbb{E}_{l}\rightarrow\mathbb{E}$
such that $\varphi\left(P_{1},...,P_{l}\right)\subseteq P$ has
property ($\mathcal{UP}$) if:

\begin{quotation}
For every convex body $Q\subset\mathbb{F}$ with $0\in\text{int}(Q)$ and every multilinear mapping  $T:\mathbb{E}_{1}\times\cdots\times\mathbb{E}_{l}\rightarrow\mathbb{F},$
if $T\left(P_{1},...,P_{l}\right)\subseteq Q,$ there exists a unique
linear mapping  $T_{\varphi}:\mathbb{E}\rightarrow\mathbb{F}$
such that $T_{\varphi}\left(P\right)\subseteq Q$ and $T_{\varphi}\circ\varphi=T.$
\end{quotation}

\begin{prop}
\label{prop: universal property of p. proj}
(Universal property of $\otimes_{\pi}$)
Let $P_{i}\subset\mathbb{E}_{i},$ $i=1,\ldots,l,$ be convex bodies containing $0$ in the interior. Then the pair ($P_{1}\otimes_{\pi}\cdots\otimes_{\pi}P_{l}$, $\otimes$) satisfies ($\mathcal{UP}$).
\end{prop}
\begin{proof}
 By definition of $\otimes_{\pi},$ $\left\{x^1\otimes\cdots\otimes x^l\in\otimes_{i=1}^l\mathbb{E}_i:x^i\in P_i\right\}\subseteq P_{1}\otimes_{\pi}\cdots\otimes_{\pi}P_{l}.$ Now, let  $T:\mathbb{E}_{1}\times\cdots\times\mathbb{E}_{l}\rightarrow\mathbb{F}$ be a multilinear mapping such that $T\left(P_{1},...,P_{l}\right)\subseteq Q$ and let $\hat{T}$ be its associated linear mapping.  The relation   $\hat{T}\circ\otimes=T$ implies that $\hat{T}\left(P_{1}\otimes_{\pi}\cdots\otimes_{\pi}P_{l}\right)\subseteq Q$.
\end{proof}

\begin{prop}(Uniqueness for the universal property)
\label{prop: * is p universal}
 Let $P_{i}\subset\mathbb{E}_{i},$ $i=1,...,l$  be convex bodies containing $0$ in the interior. If $P\subset\mathbb{E}$ is a  convex body  with $0$ in its interior, $\varphi:\mathbb{E}_{1}\times\cdots\times\mathbb{E}_{l}\rightarrow\mathbb{E}$ is a  multilinear map such that $\varphi\left(P_{1},...,P_{l}\right)\subseteq P$
and
the pair $\left(P,\varphi\right)$ satisfies
($\mathcal{UP}$), then there exists a linear isomorphism $\Psi:\mathbb{E}\rightarrow \otimes_{i=1}^{l}\mathbb{E}_{i}$
such that $\Psi\left(P\right)=P_{1}\otimes_{\pi}\cdots\otimes_{\pi}P_{l}$ and $\Psi\varphi=\otimes.$
\end{prop}
\begin{proof}
Since   $\left(P,\varphi\right)$ and $(P_{1}\otimes_{\pi}\cdots\otimes_{\pi}P_{l},\otimes)$
have the property ($\mathcal{UP}$),
there exist linear mappings $S:\mathbb{E}\rightarrow \otimes_{i=1}^{l}\mathbb{E}_{i},$
 and $T:\otimes_{i=1}^{l}\mathbb{E}_{i}\rightarrow\mathbb{E}$ such that $S\circ\varphi=\otimes,$
$T\circ\otimes=\varphi,$ $S\left(P\right)\subseteq P_{1}\otimes_{\pi}\cdots\otimes_{\pi}P_{l}$
and $T\left(P_{1}\otimes_{\pi}\cdots\otimes_{\pi}P_{l}\right)\subseteq P.$
From this, the linear map $TS:\mathbb{E}\rightarrow\mathbb{E}$ verifies:
$
TS\varphi\left(x^{1},...,x^{l}\right)=\varphi\left(x^{1},...,x^{l}\right),
$
for every $x^{i}\in\mathbb{E}_{i}.$
This  equality, together with the uniqueness of the linear extension, implies that   $TS=I_{\mathbb{E}}$.
Then,
$
P_{1}\otimes_{\pi}\cdots\otimes_{\pi}P_{l}=ST(P_{1}\otimes_{\pi}\cdots\otimes_{\pi}P_{l})\subseteq S\left(P\right)\subseteq P_{1}\otimes_{\pi}\cdots\otimes_{\pi}P_{l}.
$
Thus, $S\left(P\right)=P_{1}\otimes_{\pi}\cdots\otimes_{\pi}P_{l}$ and $S\circ\varphi=\otimes.$
This completes the proof.
\end{proof}

The next property   of $\otimes_{\pi}$ and $\otimes_{\epsilon}$  will be fundamental to what follows:
\begin{prop}
 \label{prop: uniform projective injective product}
(Uniform property of $\otimes_{\pi},\otimes_{\epsilon}$)
Let $T_i:\mathbb{E}_i\rightarrow\mathbb{F}_i,$ $i=1,\ldots,l,$ be linear maps. If $P_i\subset\mathbb{E}_i$, $Q_i\subset\mathbb{F}_i$  are convex bodies containing $0$ in the interior and $T_i\left(P_i\right)\subseteq Q_i,$ $i=1,\ldots,l,$ then
\begin{equation*}
T_1\otimes\cdots\otimes T_l\left(P_{1}\otimes_{\pi}\cdots\otimes_{\pi}P_{l}\right)\subseteq Q_{1}\otimes_{\pi}\cdots\otimes_{\pi}Q_{l}
\end{equation*}
and
\begin{equation*}
T_1\otimes\cdots\otimes T_l\left(P_{1}\otimes_{\epsilon}\cdots\otimes_{\epsilon}P_{l}\right)\subseteq Q_{1}\otimes_{\epsilon}\cdots\otimes_{\epsilon}Q_{l}.
\end{equation*}
\end{prop}
\begin{proof}
To prove the first inclusion, observe that the linearity of $T_1\otimes\cdots\otimes T_l$ and the definition of $\otimes_{\pi}$ imply:
\begin{equation}
\label{eq:uniform of proj and inj product}
T_1\otimes\cdots\otimes T_l(P_{1}\otimes_{\pi}\cdots\otimes_{\pi}P_{l})=\text{conv}\left\{T_1x^1\otimes\cdots\otimes T_lx^l\in\otimes_{i=1}^l\mathbb{F}_i:x^i\in P_i\right\}.
\end{equation}
This equality, together with $T_i\left(P_i\right)\subseteq Q_i$,  yield to the first inclusion.
To prove the second, we use  that $T_{i}\left(P_{i}\right)\subseteq Q_{i}$
implies $T_{i}^{t}\left(Q_{i}^{\circ}\right)\subseteq P_{i}^{\circ}.$  Hence, by the definition of $\otimes_\epsilon,$ for every $z\in P_{1}\otimes_{\epsilon}\cdots\otimes_{\epsilon}P_{l}$
and $y^{i}\in Q_{i}^{\circ}$ we have that
$
\left|\left\langle z,T_{1}^{t}y^{1}\otimes\cdots\otimes T_{l}^{t}y^{l}\right\rangle _{H}\right|\leq1
$
or, equivalently,
$
\left|\left\langle \left(T_{1}\otimes\cdots\otimes T_{l}\right)z,y^{1}\otimes\cdots\otimes y^{l}\right\rangle _{H}\right|\leq1.
$
Thus, $\left(T_{1}\otimes\cdots\otimes T_{l}\right)z\in Q_{1}\otimes_{\epsilon}\cdots\otimes_{\epsilon}Q_{l}$ and the proof is completed.
\end{proof}

The following result sets the injectivity of $\otimes_{\epsilon}$ and the projectivity of $\otimes_{\pi}$ as defined below, in Definition \ref{defn: injective and proj tpb}.
\begin{prop}\label{prop: pi proj eps inj} Let $P_i\subset\mathbb{E}_i,$ $i=1,\ldots,l,$ be convex bodies with $0\in int(P_i)$.
\begin{enumerate}
\item For each subspaces $M_i\subset\mathbb{E}_i$, $$(P_{1}\cap M_{1})\otimes_{\epsilon}\cdots\otimes_{\epsilon}(P_{l}\cap M_{l})=P_{1}\otimes_{\epsilon}\cdots
    \otimes_{\epsilon}P_{l}\cap\otimes_{i=1}^{l}M_{l}.$$
    \item  For any $l$-tuple of surjective linear maps  $T_{i}:\mathbb{E}_{i}\rightarrow\mathbb{F}_{i},$  $$T_{1}\otimes\cdots\otimes T_{l}\left(P_{1}\otimes_{\pi}\cdots
        \otimes_{\pi}P_{l}\right)=T_{1}
        P\otimes_{\pi}\cdots\otimes_{\pi}T_{l}P.$$
\end{enumerate}
\end{prop}
\begin{proof} That $\otimes_\epsilon$ respects sections follows by  applying the Hanh-Banach Theorem. By duality,  the projective tensor product of convex bodies $\otimes_{\pi}$ preserves quotients.
\end{proof}

In the case of $0$-symmetric convex bodies $P_i\subset\mathbb{E}_i,$ $i=1,\ldots,l,$ the products $\otimes_\pi, \otimes_\epsilon$ and the norms $\pi(\cdot), \epsilon(\cdot)$ are related by the following equalities (see \cite[Section 2]{tensorialbodies}):
\begin{equation}
\label{eq: p. inj and p.proj  unit ball inj. norm and proj norm}
P_{1}\otimes_{\pi}\cdots\otimes_{\pi}P_{l}=B_{\otimes_{\pi,i=1}^{l}\left(\mathbb{E}_{i},g_{P_{i}}\left(\cdot\right)\right)}\hspace{0.3cm}\mbox{and}\hspace{0.3cm}P_{1}\otimes_{\epsilon}\cdots\otimes_{\epsilon}P_{l}=B_{\otimes_{\epsilon,i=1}^{l}\left(\mathbb{E}_{i},g_{P_{i}}\left(\cdot\right)\right)}.
\end{equation}

The use of the terminology ``projective tensor product'' and ``injective tensor product'' for the class of  convex bodies with $0$ in the interior is, therefore,   consistent with their use for tensor norms.

\subsection{Tensorial convex bodies: definition and examples.}\label{subs: examples}

\begin{defn}
\label{def:tensorial convex bodies}
A compact convex set $Q\subset\otimes_{i=1}^{l}\mathbb{R}^{d_{i}}$ is called a \textbf{tensorial convex body in} $\mathbf{\otimes_{i=1}^{l}\mathbb{R}^{d_{i}}}$ if there exist convex bodies $Q_{i}\subset\mathbb{R}^{d_{i}},$  $i=1,...,l,$ containing zero in the interior such that
\begin{equation}
\label{eq: tensor body}
Q_{1}\otimes_{\pi}\cdots\otimes_{\pi}Q_{l}\subseteq Q\subseteq Q_{1}\otimes_{\epsilon}\cdots\otimes_{\epsilon}Q_{l}.
\end{equation}
In this case, we  say that $Q$ is a \textbf{tensorial convex body with respect to} $\mathbf{Q_1,\ldots,Q_l}.$
\end{defn}

Notice that since $0\in\text{int}(P_{1}\otimes_{\pi}\cdots\otimes_{\pi}P_{l}),$ $P$ is also a convex body containing zero in the interior. Furthermore, as a consequence of Proposition \ref{prop: duality} its polar set is also a tensorial convex body. The following  are examples of tensorial convex bodies. Details can be found in \cite{tensorialbodies}.

\begin{enumerate}
\item{\sl The injective and the projective tensor product.}
Let $P_{i}\subset\mathbb{E}_{i},$ $i=1,\ldots,l,$  be  convex bodies containing $0$ in the interior.
 Then, trivially,
$
P_{1}\otimes_{\pi}\cdots\otimes_{\pi}P_{l}$ and $
P_{1}\otimes_{\epsilon}\cdots\otimes_{\epsilon}P_{l}$ satisfy (\ref{eq: tensor body}) with respect to $P_1,\ldots P_l$. Consequently, they are tensorial convex bodies.

\item {\sl Hilbertian tensor product of ellipsoids $\otimes_2$.} If $\mathcal{E}_i=T_{i}\left(B_2^{d_i}\right),$ $i=1,\ldots,l$ are ellipsoids in $\mathbb{R}^{d_{i}},$ $i=1,...,l$ respectively, then the \textbf{Hilbertian tensor product of $\mathcal{E}_1,\ldots,\mathcal{E}_l$}, introduced in \cite{Aubrun2006}, is defined as
\begin{equation*}
\mathcal{E}_1\otimes_2\cdots\otimes_2\mathcal{E}_l:=T_1\otimes\cdots\otimes T_l\left(B_{2}^{d_{1},\ldots,d_{l}}\right).
\end{equation*}
It can be directly proved that $\mathcal{E}_1\otimes_2\cdots\otimes_2\mathcal{E}_l$ is the closed unit ball of the Hilbert tensor product $\otimes_{H,i=1}^l\left(\mathbb{R}^{d_{i}},g_{\mathcal{E}_{i}}\right)$.

\item {\sl Unit ball of $\ell_p^{d}$.}
Let $d=d_1\cdots d_l$ and
$1\leq p\leq\infty$. Then, $B_{p}^{d}=B_{p}^{d_{1},...,d_{l}}$ is a tensorial convex body in $\otimes_{i=1}^l\mathbb{R}^{d_i}$.

\end{enumerate}

\subsection{Tensor product of convex bodies with zero in the interior}\label{subsec: Tensor product of convex bodies}
In what follows we will fix  the scalar product on  $\otimes_{i=1}^{l}\mathbb{E}_{i}$ as  the one associated to the Hilbert tensor product
 $\otimes_{H,i=1}^{l}\mathbb{E}_{i}.$ In this way, if $Q\subset\otimes_{i=1}^{l}\mathbb{E}_{i}$ is a nonempty subset,  its  polar set  is $$Q^{\circ}=\left\{ z\in\otimes_{i=1}^{l}\mathbb{E}_{i}:{\sup}_{u\in Q}\left\langle u,z\right\rangle _{H}\leq1\right\}.$$

\subsubsection*{Definition and basic properties of tensor products}

\begin{defn}
\label{defn: tensor prod of cb}
A \textbf{tensor product $\otimes_{\alpha}$ of
order $l$ of convex bodies with zero in the interior} is an assignment of a
convex body $P_{1}\otimes_{\alpha}\cdots
\otimes_{\alpha}P_{l}\subset\otimes_{i=1}^{l}\mathbb{\mathbb{E}}_{i}$,
 to each $l$-tuple  $P_{i}\subset\mathbb{E}_{i},$ $i=1,...,l,$ of convex bodies containing $0$ in the interior,
such that the following conditions are satisfied:
\begin{enumerate}
\item $P_{1}\otimes_{\pi}\cdots\otimes_{\pi}P_{l}\subseteq P_{1}\otimes_{\alpha}\cdots\otimes_{\alpha}P_{l}\subseteq P_{1}\otimes_{\epsilon}\cdots\otimes_{\epsilon}P_{l},$ that is, it is a tensorial convex body with respect to $P_1,\ldots, P_l$.

\item \textit{(Uniform property)} For every linear mapping $T_{i}:\mathbb{E}_{i}\rightarrow\mathbb{F}_{i},$ $i=1,..,l$, if  $Q_i\subset\mathbb{F}_i,$  are convex bodies with $0$ in the interior and $T_i\left(P_i\right)\subseteq Q_i,$ then
\[
T_{1}\otimes\cdots\otimes T_{l}\left(P_{1}\otimes_{\alpha}\cdots\otimes_{\alpha}P_{l}\right)\subseteq Q_{1}\otimes_{\alpha}\cdots\otimes_{\alpha}Q_{l}.
\]
\end{enumerate}
\end{defn}
When no confusion can arise, we will refer to them simply as \textit{tensor products of convex bodies}. Note that $0\in\text{int}(P_{1}\otimes_{\alpha}\cdots\otimes_{\alpha}P_{l})$.
Since condition (1)  holds for $\otimes_\pi, \otimes_\epsilon,$ (see Subsection \ref{sec:inj proj tensor}), Proposition \ref{prop: uniform projective injective product} implies  that  they  are  tensor products of convex bodies.

\begin{prop}
\label{prop:linear invariance of products}
Let $\otimes_{\alpha}$ be a tensor product of convex bodies with zero in the interior and let  $P_{i}\subset\mathbb{E}_{i},$ $i=1,\ldots,l,$ be  convex bodies with $0\in int(P_i)$. Then, for every bijective linear mapping $T_{i}:\mathbb{E}_{i}\rightarrow\mathbb{F}_{i}$
we have:
\[
T_{1}\otimes\cdots\otimes T_{l}\left(P_{1}\otimes_{\alpha}\cdots\otimes_{\alpha}P_{l}\right)=
T_{1}P_1\otimes_{\alpha}\cdots\otimes_{\alpha}T_{l}P_l.
\]
\end{prop}
\begin{proof} Observe that being
 $T_{i}:\mathbb{E}_{i}\rightarrow\mathbb{F}_{i}$
 bijective, we know that $Q_{i}=T_{i}\left(P_{i}\right)$
is a convex body containing $0$ in the interior. By the uniform property of $\otimes_{\alpha},$
we have
$
T_{1}\otimes\cdots\otimes T_{l}\left(P_{1}\otimes_{\alpha}\cdots\otimes_{\alpha}P_{l}\right)\subseteq T_{1}P_1\otimes_{\alpha}\cdots\otimes_{\alpha}T_{l}P_l
$
and
$
T_{1}^{-1}\otimes\cdots\otimes T_{l}^{-1}\left(Q_{1}\otimes_{\alpha}\cdots\otimes_{\alpha}Q_{l}\right)\subseteq T_{1}^{-1}Q_{1}\otimes_{\alpha}\cdots\otimes_{\alpha}T_{l}^{-1}Q_{l}.
$
Since  $T_{i}^{-1}(Q_{i})=P_i$, we have that  $T_{1}\otimes\cdots\otimes T_{l}\left(P_{1}\otimes_{\alpha}\cdots\otimes_{\alpha}P_{l}\right)
=T_{1}P_1\otimes_{\alpha}\cdots\otimes_{\alpha}T_{l}P_l.$
\end{proof}

%As a consequence of (\ref{eq: p. inj and p.proj  unit ball inj. norm and proj norm}) and Corollary \ref{cor:tensor products preserve symmetric convex bodies}, if $P_i\subset\mathbb{E}_i,$ $i=1,\ldots,l$, are $0$-symmetric convex bodies with associated Minkowski functional $g_{P_i},$ then $g_{P_{1}\otimes_{\alpha}\cdots\otimes_{\alpha}P_{l}}\left(\cdot\right)$ is a reasonable crossnorm on $\otimes_{i=1}^l\left(\mathbb{E}_i,g_{P_i}\right).$  Thus, $P_{1}\otimes_{\alpha}\cdots\otimes_{\alpha}P_{l}$ is a tensorial body (see \cite[Definition 3.3]{tensorialbodies}), for every tensor product of convex bodies $\otimes_{\alpha}$. This property as well as the connection with tensor norms will be detailed in Section 4.

\subsubsection*{Polarity on tensor products of convex bodies}

The relation of polarity between $\otimes_{\epsilon}$ and $\otimes_{\pi}$ determined by the definition of $\otimes_\epsilon$ can be formulated in the following equivalent ways:
\begin{prop}
\label{prop: duality} Let $P_{i}\subset\mathbb{E}_{i},$ $i=1,\ldots,l$ be   convex bodies with $0\in\text{int}(P_i)$. Then,
\begin{enumerate}
  \item $
(P_{1}\otimes_{\epsilon}\cdots\otimes_{\epsilon}P_{l})^{\circ}=
P_{1}^{\circ}\otimes_{\pi}\cdots\otimes_{\pi}P_{l}^{\circ}.
$

\item $ (P_{1}\otimes_{\pi}\cdots\otimes_{\pi}P_{l})^{\circ}=
P_{1}^{\circ}\otimes_{\epsilon}\cdots\otimes_{\epsilon}P_{l}^{\circ}.
$
\end{enumerate}
\end{prop}

Thus, according to the following definition, $\otimes_{\epsilon}$ and $\otimes_{\pi}$  are dual tensor products:
\begin{defn}
\label{defn: dual tensor prod of cb} Let $\otimes_{\alpha}$ be a tensor product of convex bodies with zero in the interior. Its  \textbf{ dual tensor product} $\boldsymbol{\otimes_{\alpha^{\prime}}}$  is the tensor product that assigns, to  each $l$-tuple  $P_{i}\subset\mathbb{E}_{i},$ $i=1,...,l,$ of convex bodies with $0\in int(P_i)$,   the
convex body   \begin{equation}
\label{eqn: dual tensor product}P_{1}\otimes_{\alpha^{\prime}}\cdots\otimes_{\alpha^{\prime}}P_{l}:=\left(P_{1}^{\circ}\otimes_{\alpha}\cdots
\otimes_{\alpha}P_{l}^{\circ}\right)^{\circ}.
\end{equation}
\end{defn}
Now we check  that  $\boldsymbol{\otimes_{\alpha^{\prime}}}$  is well defined. Since $0\in P_{1}^{\circ}\otimes_{\alpha}\cdots\otimes_{\alpha}P_{l}^{\circ}$, then $0\in  (P_{1}^{\circ}\otimes_{\alpha}\cdots
\otimes_{\alpha}P_{l}^{\circ})^{\circ}$ (see \cite[Theorem 1.6.1.]{Schneider1993}). Condition  (1) in  Definition \ref{defn: tensor prod of cb} follows using  such  condition  for $P_1^{\circ}\otimes_{\alpha}\cdots\otimes_{\alpha}P_l^{\circ}$. Then, by taking polars the following inclusions are obtained:
\[
\left(P_{1}^{\circ}\otimes_{\epsilon}\cdots\otimes_{\epsilon}P_{l}^{\circ}\right)^{\circ}
\subseteq\left(P_{1}^{\circ}\otimes_{\alpha}\cdots
\otimes_{\alpha}P_{l}^{\circ}\right)^{\circ}
\subseteq\left(P_{1}^{\circ}\otimes_{\pi}\cdots
\otimes_{\pi}P_{l}^{\circ}\right)^{\circ}.
\]
 By Proposition \ref{prop: duality}, this is the same as:
\begin{equation*}
P_{1}\otimes_{\pi}\cdots\otimes_{\pi}P_{l}\subseteq P_{1}\otimes_{\alpha^{\prime}}\cdots\otimes_{\alpha^{\prime}}P_{l}\subseteq P_{1}\otimes_{\epsilon}\cdots\otimes_{\epsilon}P_{l}.
\end{equation*}
To prove condition (2), consider $T_{i}:\mathbb{E}_{i}\rightarrow\mathbb{F}_{i}$
such that $T_{i}\left(P_{i}\right)\subseteq Q_{i}$, $i=1,...,l$.  Using  the uniform property of $\otimes_{\alpha}$, we get:
\[
T_{1}^{t}\otimes\cdots\otimes T_{l}^{t}\left(Q_{1}^{\circ}\otimes_{\alpha}\cdots\otimes_{\alpha}Q_{l}^{\circ}\right)\subseteq P_{1}^{\circ}\otimes_{\alpha}\cdots\otimes_{\alpha}P_{l}^{\circ}.
\]
Thus,
$\left(T_{1}^{t}\otimes\cdots\otimes T_{l}^{t}\right)^{t}\left(P_{1}^{\circ}\otimes_{\alpha}\cdots
\otimes_{\alpha}P_{l}^{\circ}\right)^{\circ}\subset\left(Q_{1}^{\circ}
\otimes_{\alpha}\cdots\otimes_{\alpha}Q_{l}^{\circ}\right)^{\circ}.$ That is,
$$ T_{1}\otimes\cdots\otimes T_{l}\left(P_{1}\otimes_{\alpha^{\prime}}\cdots\otimes_{\alpha^{\prime}}P_{l}\right) \subseteq Q_{1}\otimes_{\alpha^{\prime}}\cdots\otimes_{\alpha^{\prime}}Q_{l}.
$$
\qed

The following proposition collects the basic properties of $\otimes_{\alpha^{\prime}}.$ We do not include its proof, since it is a direct consequence of the definition of $\otimes_{\alpha^{\prime}}$ and Proposition \ref{prop: duality}.

\begin{prop}
\label{prop:property dual product}
The following hold:
\begin{enumerate}
\item The dual of $\otimes_{\pi}$ (resp. $\otimes_{\epsilon}$) is $\otimes_{\epsilon}$ (resp. $\otimes_{\pi}$).
\item If $\otimes_{\alpha}$ is a tensor product of convex bodies, then $\otimes_{\alpha^{\prime\prime}}=\otimes_{\alpha}.$
\item  If $\otimes_{\alpha},\otimes_{\beta}$ are tensor products of convex bodies such that $P_1\otimes_{\alpha}\cdots\otimes_{\alpha}P_l\subseteq P_1\otimes_{\beta}\cdots\otimes_{\beta}P_l$ for all convex bodies $P_i\subset\mathbb{E}_i,$ $0\in\text{int}(P_i),$ then $P_1\otimes_{\beta^{\prime}}\cdots\otimes_{\beta^{\prime}}P_l\subseteq P_1\otimes_{\alpha^{\prime}}\cdots\otimes_{\alpha^{\prime}}P_l.$
\end{enumerate}
\end{prop}

\subsubsection*{Sections and quotients: injective and projective  properties}

\begin{defn}
\label{defn: injective and proj tpb} We say that
a  tensor product  of
convex bodies with zero in the interior $\otimes_{\alpha}$ is
                 \textbf{\textsl{injective}} if for each $P_{i}\subset\mathbb{E}_{i}$ with $0\in\text{int}(P_i)$
and each subspace  $M_{i}\subseteq\mathbb{E}_{i},$ $i=1,...,l,$ it holds that
\begin{equation}\label{eq: def injectivity}
\left(P_{1}\cap M_{1}\right)\otimes_{\alpha}\cdots\otimes_{\alpha}\left(P_{l}\cap M_{l}\right)=\left(P_{1}\otimes_{\alpha}\cdots\otimes_{\alpha}P_{l}\right)\cap\otimes_{i=1}^{l}M_{i}.
\end{equation}
Here, the scalar product on the space $M_{i}$ is the one induced
by $\mathbb{E}_{i}.$ In this way each $P_{i}\cap M_{i}$
is a convex body with $0$ in the interior.

We say that $\otimes_{\alpha}$  is
\textbf{\textsl{projective}} if for each $P_{i}\subset\mathbb{E}_{i},$ $0\in\text{int}(P_i),$
and every surjective linear mapping $T_{i}:\mathbb{E}_{i}\rightarrow\mathbb{F}_{i},$
$i=1,...,l,$  it holds that
\begin{equation}\label{eq: def projectivity}
T_{1}\otimes\cdots\otimes T_{l}\left(P_{1}\otimes_{\alpha}\cdots\otimes_{\alpha}P_{l}\right)
=T_{1}P_1\otimes_{\alpha}\cdots\otimes_{\alpha}T_{l}P_l.
\end{equation}
\end{defn}

Injectivity and projectivity are dual properties, in the following sense:

\begin{thm}
\label{thm:duality inj and proj}
Let $\otimes_{\alpha}$ be a tensor product of convex
bodies, then $\otimes_{\alpha}$ is projective if and only if $\otimes_{\alpha^{\prime}}$
is injective.
\end{thm}

\begin{proof} Suppose that $\otimes_{\alpha}$ is projective. For each $j$, let $M_j\subset \mathbb{E}_{j}$
be a subspace and  let  $L_j$ be the orthogonal projection onto $M_j.$ For each $P_{j}\subset\mathbb{E}_{j},$ $0\in\text{int}(P_j),$
$\left(P_{j}\cap M_{j}\right)^{\diamond}$ denotes the polar set of $P_{j}\cap M_{j}\subseteq M_{j}$ determined by the scalar product on $M_j$ induced by $\mathbb{E}_j.$

Using  the relation \cite[(1.13)]{Aubrun2017}  several times and the projectivity of $\otimes_{\alpha}$, we have
\begin{multline*}
  \left(P_{1}\cap M_{1}\right)\otimes_{\alpha^{\prime}}\cdots\otimes_{\alpha^{\prime}}
  \left(P_{l}\cap M_{l}\right) =\left(\left(P_{1}\cap M_{1}\right)^{\diamond}\otimes_{\alpha}\cdots\otimes_{\alpha}\left(P_{l}\cap M_{l}\right)^{\diamond}\right)^{\diamond}=
\\
  (L_1(P_{1}^{\circ})\otimes_{\alpha}\cdots\otimes_{\alpha}
  L_l(P_{1}^{\circ}))^{\diamond}= (L_1\otimes\cdots\otimes L_l\left(P_{1}^{\circ}
\otimes_{\alpha}\cdots\otimes_{\alpha}P_{l}^{\circ}\right))^{\diamond}
=\\ (\left(P_{1}^{\circ}\otimes_{\alpha}\cdots
  \otimes_{\alpha}P_{l}^{\circ}\right)^{\circ}
  \cap\otimes_{j=1}^{l}M_{j})^{\diamond\diamond}=P_{1}
  \otimes_{\alpha^{\prime}}\cdots\otimes_{\alpha^{\prime}}
  P_{l}\cap\otimes_{j=1}^{l}M_{j}.
\end{multline*}
Consequently,  $\otimes_{\alpha^{\prime}}$ is injective.

To prove the reciprocal, we will use the following relation: if $T:\mathbb{E}\rightarrow \mathbb{F}$ is a linear map and $P\subset \mathbb{E}$ is a convex set then,
\begin{equation}
\label{eq:transpose of polar TA}
T^{t}\left(\left(TP\right)^{\circ}\right)=P^{\circ}\cap T^{t}\left(\mathbb{F}\right).
\end{equation}
%It is obtained by applying $T^t$ to {\color{red}{buscar otra referencia,  Bourbaki TVS II.47 Proposition 6}}.

Let $T_{j}:\mathbb{E}_{j}\rightarrow\mathbb{F}_{j}$,  $j=1,...,l,$ be  surjective mappings. In such case, $T_1\otimes\cdots\otimes T_l$ is also surjective. The mappings $T_{j}^t$ and $(T_1\otimes\cdots\otimes T_l)^t=T_1^t\otimes\cdots\otimes T_l^t$ are injective.

Let us assume that  $\otimes_{\alpha^{\prime}}$ is injective. Using the injectivity of $\otimes_{\alpha^{\prime}}$ and applying (\ref{eq:transpose of polar TA}), we have:
%\noindent Since each $T_{j}$ is surjective we know that: $T_{j}P_{j}\in\mathcal{B}\left(\mathbb{F}\right),$
%$T_{j}^{t}$ is injective and
%\[
%T_{j}^{t}\left(\left(T_{j}P_{j}\right)^{\circ}\right)=P_{j}^{\circ}\cap T_{j}^{t}\left(\mathbb{F}_{j}\right).
%\]
%\begin{align*}
%T_{1}^{t}\left(\left(T_{1}P_{l}\right)^{\circ}\right)\otimes_{\alpha^{\prime}}\cdots\otimes_{\alpha^{\prime}}T_{l}^{t}\left(\left(T_{l}P_{l}\right)^{\circ}\right) & =P_{1}^{\circ}\cap T_{1}^{t}\left(\mathbb{F}_{1}\right)\otimes_{\alpha^{\prime}}\cdots\otimes_{\alpha^{\prime}}P_{l}^{\circ}\cap T_{l}^{t}\left(\mathbb{F}_{l}\right)\\
% & =\left(P_{1}^{\circ}\otimes_{\alpha^{\prime}}\cdots\otimes_{\alpha^{\prime}}P_{l}^{\circ}\right)\cap\otimes_{j=1}^{l}T_{j}^{t}\left(\mathbb{F}_{j}\right)\\
% & =\left(P_{1}\otimes_{\alpha}\cdots\otimes_{\alpha}P_{l}\right)^{\circ}\cap\otimes_{j=1}^{l}T_{j}^{t}\left(\mathbb{F}_{j}\right).
%\end{align*}
%Since $T_{1}\otimes\cdots\otimes T_{l}$ is surjective, then (\ref{eq:transpose of polar TA}) and the previous equality yield to:
%\[
%T_{1}^{t}\otimes\cdots\otimes T_{l}^{t}\left(\left(T_{1}\otimes\cdots\otimes T_{l}\left(P_{1}\otimes_{\alpha}\cdots\otimes_{\alpha}P_{l}\right)\right)^{\circ}\right)=\left(P_{1}\otimes_{\alpha}\cdots\otimes_{\alpha}P_{l}\right)^{\circ}\cap\otimes_{j=1}^{l}T_{j}^{t}\left(\mathbb{F}_{j}\right).
%\]
%Hence,
\begin{multline*}
T_{1}^{t}\otimes\cdots\otimes T_{l}^{t}\left(\left(T_{1}\otimes\cdots\otimes T_{l}\left(P_{1}\otimes_{\alpha}\cdots\otimes_{\alpha}P_{l}\right)\right)^{\circ}\right)=(P_1\otimes_\alpha\cdots\otimes_\alpha P_l)^\circ\cap\otimes_{i=1}^lT^t(\mathbb{F}_i)\\=P_1^\circ\otimes_{\alpha^{\prime}}\cdots\otimes_{\alpha^{\prime}}P_l^\circ\cap\otimes_{i=1}^lT^t(\mathbb{F}_i)=P_{1}^{\circ}\cap T_{1}^{t}\left(\mathbb{F}_{1}\right)\otimes_{\alpha^{\prime}}\cdots\otimes_{\alpha^{\prime}}P_{l}^{\circ}\cap T_{l}^{t}\left(\mathbb{F}_{l}\right)\\=T_{1}^{t}\left(\left(T_{1}P_{l}\right)^{\circ}\right)\otimes_{\alpha^{\prime}}\cdots\otimes_{\alpha^{\prime}}T_{l}^{t}\left(\left(T_{l}P_{l}\right)^{\circ}\right)\overset{*}{=}T_{1}^{t}\otimes\cdots\otimes T_{l}^{t}\left(\left(T_{1}P_{1}\right)^{\circ}\otimes_{\alpha^{\prime}}\cdots\otimes_{\alpha^{\prime}}\left(T_{1}P_{l}\right)^{\circ}\right).
\end{multline*}
(*) is due to the injectivity of $T_{1}^{t}\otimes\cdots\otimes T_{l}^{t}$ and Proposition \ref{prop:linear invariance of products}.
%\begin{equation}
%T_{1}^{t}\otimes\cdots\otimes T_{l}^{t}\left(\left(T_{1}\otimes\cdots\otimes T_{l}\left(P_{1}\otimes_{\alpha}\cdots\otimes_{\alpha}P_{l}\right)\right)^{\circ}\right)=T_{1}^{t}\left(\left(T_{1}P_{1}\right)^{\circ}\right)\otimes_{\alpha^{\prime}}\cdots\otimes_{\alpha^{\prime}}T_{l}^{t}\left(\left(T_{l}P_{l}\right)^{\circ}\right)
%\label{eq:for projec iff inje}
%\end{equation}
%Now, for each $j$, denote by $S_{j}:\mathbb{F}_{j}  \rightarrow T_{j}^{t}\left(\mathbb{F}_{j}\right)$ to the map sending $y^{j}  \rightarrow T_{j}^{t}\left(y^{j}\right),$
%\begin{align*}
%S_{j}:\mathbb{F}_{j} & \rightarrow T_{j}^{t}\left(\mathbb{F}_{j}\right)\\
%y^{j} & \rightarrow T_{j}^{t}\left(y^{j}\right).
%\end{align*}
%then $S_{j}$ is a bijective linear map. From Proposition \ref{prop:linear invariance of products},
%it follows:
%\[
%S_{1}\otimes\cdots\otimes S_{l}\left(\left(T_{1}P_{1}\right)^{\circ}\otimes_{\alpha^{\prime}}\cdots\otimes_{\alpha^{\prime}}\left(T_{1}P_{l}\right)^{\circ}\right)=S_{1}\left(\left(T_{1}P_{1}\right)^{\circ}\right)\otimes_{\alpha^{\prime}}\cdots\otimes_{\alpha^{\prime}}S_{l}\left(\left(T_{l}P_{l}\right)^{\circ}\right),
%\]
%which, by the injectivity of $\otimes_{\alpha^{\prime}},$ is equivalent to
%\[
%T_{1}^{t}\otimes\cdots\otimes T_{l}^{t}\left(\left(T_{1}P_{1}\right)^{\circ}\otimes_{\alpha^{\prime}}\cdots\otimes_{\alpha^{\prime}}\left(T_{1}P_{l}\right)^{\circ}\right)=T_{1}^{t}\left(\left(T_{1}P_{1}\right)^{\circ}\right)\otimes_{\alpha^{\prime}}\cdots\otimes_{\alpha^{\prime}}T_{l}^{t}\left(\left(T_{l}P_{l}\right)^{\circ}\right).
%\]
Indeed, we must have:
\[
(\left(T_{1}\otimes\cdots\otimes T_{l}\left(P_{1}\otimes_{\alpha}\cdots\otimes_{\alpha}P_{l}\right)\right)^{\circ}=\left(T_{1}P_{1}\right)^{\circ}\otimes_{\alpha^{\prime}}\cdots\otimes_{\alpha^{\prime}}\left(T_{l}P_{l}\right)^{\circ}.
\]
Finally, the result follows directly by taking polars in the previous equality and using (\ref{eqn: dual tensor product}).
%\begin{align*}
%\left(T_{1}P_{1}\right)^{\circ}\otimes_{\alpha^{\prime}}\cdots\otimes_{\alpha^{\prime}}\left(T_{l}P_{l}\right)^{\circ} & =\left(T_{1}\otimes\cdots\otimes T_{l}\left(P_{1}\otimes_{\alpha}\cdots\otimes_{\alpha}P_{l}\right)\right)^{\circ}\\
%\left(T_{1}P_{1}\otimes_{\alpha}\cdots\otimes_{\alpha}T_{l}P_{l}\right)^{\circ} & =\left(T_{1}\otimes\cdots\otimes T_{l}\left(P_{1}\otimes_{\alpha}\cdots\otimes_{\alpha}P_{l}\right)\right)^{\circ}\\
%\left(T_{1}P_{1}\otimes_{\alpha}\cdots\otimes_{\alpha}T_{l}P_{l}\right)^{\circ\circ} & =\left(T_{1}\otimes\cdots\otimes T_{l}\left(P_{1}\otimes_{\alpha}\cdots\otimes_{\alpha}P_{l}\right)\right)^{\circ\circ}\\
%T_{1}P_{1}\otimes_{\alpha}\cdots\otimes_{\alpha}T_{l}P_{l} & =T_{1}\otimes\cdots\otimes T_{l}\left(P_{1}\otimes_{\alpha}\cdots\otimes_{\alpha}P_{l}\right).
%\end{align*}
%This proves that $\otimes_{\alpha}$ is projective.
\end{proof}

\begin{rem}\label{rmk: proj is not inj}
The projective tensor product of convex bodies $\otimes_\pi$, which is projective (see Proposition \ref{prop: pi proj eps inj})  is not injective.   To see this, recall that there exist finite dimensional Banach spaces  $M\subset E$, $N$ and an element in $\mathcal{L}(M;N^{*})$ such that each of its extensions in $\mathcal{L}(E;N^{*})$ has    norm strictly greater \cite[Theorem 4]{Kakutani1939}.   This means, by the Hahn-Banach theorem and by the fundamental relation  $\mathcal{L}(E;N^{*})=(E\otimes_\pi N)^{*}$ \cite[pp. 27]{defantfloret}, that   $(B_E\cap M)\otimes_\pi B_{N}\neq(B_E \otimes_\pi B_N)\cap M\otimes N,$ where $B_E,$ $B_M$ are the closed unit balls of $E,N.$

By duality, we have that  the injective tensor product of convex bodies is not projective.
\end{rem}

\subsubsection*{The closest injective and projective tensor products}\label{subs: inj proj hulls}

The following relation on tensor products defines an order: let  $\otimes_\alpha$ and $\otimes_\beta$ be tensor products of convex bodies. We say that $\otimes_\alpha\subseteq\otimes_\beta$ if and only if for every tuple of convex bodies $P_i\subset\mathbb{E}_i,$ $i=1,\ldots,l,$ containing $0$ in their interior,  $P_{1}\otimes_{\alpha}\cdots\otimes_{\alpha}P_{l}\subseteq P_{1}\otimes_{\beta}\cdots\otimes_{\beta}P_{l}.$ We say that  $\otimes_\alpha=\otimes_\beta$ if and only if $\otimes_\alpha\subseteq\otimes_\beta$ and $\otimes_\beta\subseteq\otimes_\alpha$ hold. In this order, given any
 tensor product  of convex bodies $\otimes_{\alpha}$, it holds that  $\otimes_{\pi}\subseteq\otimes_{\alpha}\subseteq\otimes_{\epsilon}$.

%Below we will see that given a tensor product of convex bodies $\otimes_\alpha,$ we always can contruct the biggest projective tensor product below it. Similarly, we can construct the smallest injective tensor product of convex bodies above $\otimes_\alpha$.

 Given a tensor product of convex bodies
 $\otimes_{\alpha}$, let $\otimes_{\alpha^{inj}}$   be defined on  each tuple of convex bodies $P_i\subset\mathbb{E}_i,$ $0\in\text{int}(P_i),$  $i=1,\ldots,l,$ as:
\begin{equation}\label{eq: injective hull}
P_1\otimes_{\alpha^{inj}}\cdots\otimes_{\alpha^{inj}}P_l:=\cap\left\{P_1\otimes_{\beta}\cdots\otimes_{\beta}P_l:\otimes_{\beta}\text{ is injective and }\otimes_{\alpha}\subseteq\otimes_{\beta}\right\}.
\end{equation}

\begin{prop}
\label{prop:closest injective product}
Let $\otimes_{\alpha}$ be a tensor product of convex bodies, then $\otimes_{\alpha^{inj}}$ is the smallest injective tensor product of convex bodies such that $\otimes_{\alpha}\subseteq\otimes_{\alpha^{inj}}$.
\end{prop}

\begin{proof} Observe first that $\otimes_{\alpha^{inj}}$ is well defined, since $\otimes_{\alpha}\subseteq\otimes_{\epsilon}$ and   $\otimes_\epsilon$ is injective. Let   $P_i\subset\mathbb{E}_i,$ $0\in\text{int}(P_i),$ $i=1,\ldots,l,$  be any $l$-tuple of convex bodies. Then,  $P_1\otimes_{\alpha^{inj}}\cdots\otimes_{\alpha^{inj}}P_l$ is a compact convex set such that
\begin{equation*}
 P_{1}\otimes_{\pi}\cdots\otimes_{\pi}P_{l}\subseteq P_{1}\otimes_{\alpha^{inj}}\cdots\otimes_{\alpha^{inj}}P_{l}\subseteq P_{1}\otimes_{\epsilon}\cdots\otimes_{\epsilon}P_{l}.
\end{equation*}
Since $0\in int (P_{1}\otimes_{\pi}\cdots\otimes_{\pi}P_{l})$, the same holds for $P_{1}\otimes_{\alpha^{inj}}\cdots\otimes_{\alpha^{inj}}P_{l}.$ The uniform property (condition $(2)$ in Definition \ref{defn: tensor prod of cb}) is satisfied because each $\otimes_{\beta}$ satisfies it.

To see that $\otimes_{\alpha^{inj}}$ is injective, consider any subspaces $M_{i}\subseteq\mathbb{E}_{i},$ $i=1,...,l,$. Then, using that  each $\otimes_\beta$ is injective, we have:
\begin{multline*}
P_{1}\cap M_{1}\otimes_{\alpha^{inj}}\cdots\otimes_{\alpha^{inj}}P_{l}\cap M_{l} = \cap\left\{P_{1}\cap M_{1}\otimes_{\beta}\cdots\otimes_{\beta}P_{l}\cap M_{l}:\otimes_{\alpha}\subseteq\otimes_{\beta}\right\}= \\  \cap\left\{P_{1}\otimes_{\beta}\cdots\otimes_{\beta}P_{l}\cap\otimes_{i=1}^l M_i:\otimes_{\alpha}\subseteq\otimes_{\beta}\right\}= P_{1}\otimes_{\alpha^{inj}}\cdots\otimes_{\alpha^{inj}}P_{l}\cap\otimes_{i=1}^l M_i.
\end{multline*}
Finally, being $\otimes_{\alpha^{inj}}$ an injective tensor product, it has to be the smallest one bigger than $\otimes_{\alpha}.$
\end{proof}

%
%Since taking duals reverses inclusions between tensor products of convex bodies (see (3) in Proposition \ref{prop:property dual product}),  by using the duality relation between projective and injective tensor products of convex bodies (see Theorem \ref{thm:duality inj and proj}) and Proposition \ref{prop:closest injective product} above, we can define the closest projective tensor product, for a g
Given a tensor product of convex bodies $\otimes_\alpha,$
let $\otimes_{\alpha^{proj}}$ be defined on  each  tuple of convex bodies $P_i\subset\mathbb{E}_i,$ $0\in\text{int}(P_i),$  $i=1,\ldots,l,$ $\otimes_{\alpha^{proj}}$ as:
\begin{equation}\label{eq: projective hull}
P_1\otimes_{\alpha^{proj}}\cdots\otimes_{\alpha^{proj}}P_l:=P_1\otimes_{((\alpha^{\prime})^{inj})^{\prime}}\cdots\otimes_{((\alpha^{\prime})^{inj})^{\prime}}P_l.
\end{equation}

\begin{prop}
\label{prop:projective associate}
For a given  $\otimes_{\alpha}$, $\otimes_{\alpha^{proj}}$ is the biggest projective tensor product of convex bodies such that $\otimes_{\alpha^{proj}}\subseteq\otimes_{\alpha}$. Furthermore,
\[
P_1\otimes_{\alpha^{proj}}\cdots\otimes_{\alpha^{proj}}P_l=\overline{\text{conv}}(\cup\left\{P_1\otimes_{\beta}\cdots\otimes_{\beta}P_l:\otimes_{\beta}\text{ is projective, }\otimes_{\beta}\subseteq\otimes_{\alpha}\right\}).
\]
\end{prop}
\begin{proof}
By Theorem \ref{thm:duality inj and proj} and Proposition \ref{prop:closest injective product}, $\otimes_{\alpha^{proj}}$ is a projective tensor product of convex bodies. Then, from (2) and (3) in Proposition \ref{prop:property dual product}, it follows easily that $\otimes_{\alpha^{proj}}$ is the biggest projective tensor product of convex bodies below $\otimes_{\alpha}.$

To prove the second part, let $P_i\subset\mathbb{E}_i,$ $i=1,\ldots,l$ be convex bodies containing $0$ in the interior. Then, by Theorem \ref{thm:duality inj and proj}, we have:
\begin{multline*}
P_1\otimes_{\alpha^{proj}}\cdots\otimes_{\alpha^{proj}}P_l =P_1\otimes_{((\alpha^{\prime})^{inj})^{\prime}}\cdots
\otimes_{((\alpha^{\prime})^{inj})^{\prime}}P_l
=(P_1^{\circ}\otimes_{(\alpha^{\prime})^{inj}}\cdots
\otimes_{(\alpha^{\prime})^{inj}}P_l^{\circ})^{\circ}=\\
\left(\underset{\stackrel{\otimes_{\beta}}{{inj}}}{\cap}\left\{ P_1^{\circ}\otimes_{\beta}\cdots
\otimes_{\beta}P_l^{\circ}:\otimes_{\alpha^{\prime}}
\subseteq\otimes_{\beta}\right\}\right)^{\circ}
\overset{*}{=}\overline{\text{conv}}\left(\underset{\stackrel{\otimes_{\beta}}{
\text{inj}}}{\cup}\left\{ (P_1^{\circ}\otimes_{\beta}\cdots
\otimes_{\beta}P_l^{\circ})^{\circ}:\otimes_{\alpha^{\prime}}
\subseteq\otimes_{\beta}\right\}\right)=\\
\overline{\text{conv}}\left(\underset{\stackrel{\otimes_{\beta^{\prime}}}{
\text{proj}}}{\cup}\left\{ P_1\otimes_{\beta^{\prime}}\cdots\otimes_{\beta^{\prime}}
P_l:\otimes_{\beta^{\prime}}\subseteq
\otimes_{\alpha^{\prime\prime}}\right\}\right)
\overset{**}{=}\overline{\text{conv}}\left(\underset{
\stackrel{\otimes_{\gamma}}{\text{proj}}}{\cup}\left\{ P_1\otimes_{\gamma}\cdots\otimes_{\gamma}P_l:\otimes_{\gamma}\subseteq\otimes_{\alpha}\right\}\right).
\end{multline*}
*It is  used that  $(\cap_{i\in I}K_i)^{\circ}=\overline{\text{conv}}\cup_{i\in I}K_i^{\circ}$ (see \cite[Theorem 1.6.2]{Schneider1993}).

\noindent**It is used (2) in Proposition \ref{prop:property dual product}.
\end{proof}

\section{Tensor products on finite dimensional Banach spaces and tensor products of $0$-symmetric convex bodies: a bijective correspondence.}\label{sec: tensor product 0symm}
We apply the results in the previous  section to the case of
$0$-symmetric convex bodies. The  theory of tensor products   thus derived is  consistent with the theory of tensor products of finite dimensional Banach spaces. The precise statement is written in Theorem \ref{thm: main theorem the bijection}.

 \begin{defn}
\label{defn: tensor prod of 0symm cb}
A \textbf{tensor product $\otimes_{\alpha}$ of
order $l$ of $0$-symmetric  convex bodies} is an assignment of a
$0$-symmetric convex body $P_{1}\otimes_{\alpha}\cdots
\otimes_{\alpha}P_{l}\subset\otimes_{i=1}^{l}\mathbb{\mathbb{E}}_{i}$,
 to each $l$-tuple  $P_{i}\subset\mathbb{E}_{i},$ $i=1,...,l,$ of $0$-symmetric convex bodies,
such that the following conditions are satisfied:
\begin{enumerate}
\item $P_{1}\otimes_{\pi}\cdots\otimes_{\pi}P_{l}\subseteq P_{1}\otimes_{\alpha}\cdots\otimes_{\alpha}P_{l}\subseteq P_{1}\otimes_{\epsilon}\cdots\otimes_{\epsilon}P_{l}.$

\item \textit{(Uniform property)} For every linear mapping $T_{i}:\mathbb{E}_{i}\rightarrow\mathbb{F}_{i},$ $i=1,..,l$, if  $Q_i\subset\mathbb{F}_i,$  are $0$-symmetric convex bodies and $T_i\left(P_i\right)\subseteq Q_i,$ then
\[
T_{1}\otimes\cdots\otimes T_{l}\left(P_{1}\otimes_{\alpha}\cdots\otimes_{\alpha}P_{l}\right)\subseteq Q_{1}\otimes_{\alpha}\cdots\otimes_{\alpha}Q_{l}.
\]
\end{enumerate}
where  $\otimes_{\pi}$ and $\otimes_{\epsilon}$ are as in Section \ref{sec:inj proj tensor}.
\end{defn}

It can be directly checked that  $\otimes_{\pi}$ and $\otimes_{\epsilon}$ satisfy Definition \ref{defn: tensor prod of 0symm cb}. It also follows from the following more general result:

\begin{prop}\label{prop: symetry of P} Let $\otimes_{\alpha}$ be a tensor product of convex bodies.
If  at least one of the convex bodies with $0$ in its interior $P_i\subset \mathbb{E}_i$, $i=1,\ldots,l$ is  $0$-symmetric, then
$P_{1}\otimes_{\alpha}\cdots\otimes_{\alpha}P_{l}$ is a $0$-symmetric convex body.
\end{prop}
\begin{proof}  Assume, w.l.o.g. that $P_1\subset \mathbb{E}_1$,  is $0$-symmetric. Consider the identity map on $\mathbb{E}_i$, $I_i=I_{\mathbb{E}_i}$ for $i=1,\ldots,l.$ By Proposition \ref{prop:linear invariance of products} $(-I_1)\otimes\cdots\otimes I_{l}(P_{1}\otimes_{\alpha}\cdots\otimes_{\alpha}P_{l})=(-P_{1})\otimes_{\alpha}\cdots\otimes_{\alpha}P_{l}.$ Since $P_1=-P_1$, we have $-(P_{1}\otimes_{\alpha}\cdots\otimes_{\alpha}P_{l})=P_{1}\otimes_{\alpha}\cdots\otimes_{\alpha}P_{l}$,  as required.
\end{proof}

\begin{cor}
\label{coro: tensor induce tensor 0-symm}
%(Linear invariance of $\otimes_\alpha$)
Every  tensor product of convex bodies  induces a tensor product of $0$-symmetric convex bodies.
\end{cor}

\subsection{Dual tensor product, injectivity and projectivity}\label{subs: analogue for 0-symm}
Using Corollary  \ref{coro: tensor induce tensor 0-symm}, it follows that the notions and  results  of  the previous section have analogues for $0$-symmetric convex bodies. We will use them, but we omit  the proofs.
Concretely, we say that $\otimes_{\alpha}$ is {\bf injective} if it satisfies (\ref{eq: def injectivity}) and {\bf projective} if it satisfies (\ref{eq: def projectivity}) for $0$-symmetric convex bodies.

\begin{prop}  Let $\otimes_{\alpha}$ be a tensor product of $0$-symmetric convex bodies. Then,
\begin{enumerate}
  \item Relation (\ref{eqn: dual tensor product}) defines a  tensor product of $0$-symmetric convex bodies $\boldsymbol{\otimes_{\alpha^{\prime}}}$. It will be called the  {\bf dual tensor product} of  $\otimes_{\alpha}$.
  \item $\otimes_{\alpha}$ is injective if and only if $\otimes_{\alpha^{\prime}}$ is projective.
\end{enumerate}
\end{prop}

 Proposition \ref{prop:property dual product}, as well as the definitions and results in  Subsection \ref{subs: inj proj hulls} concerning injective $\otimes_{\alpha^{inj}}$ and projective $\otimes_{\alpha^{proj}}$ hulls,  also have analogues.

\subsection{Relation with the Banach-Mazur distance}\label{subs: Banach-Mazur dist}
We will see now that
a  tensor product of $0$-symmetric convex bodies is uniformly continuous with respect to the Banach-Mazur distance. Recall that the Banach-Mazur distance
% between isomorphic Banach spaces $X$ and $Y$ is defined as:
%\[
%\delta^{BM}\left(X,Y\right):=\inf\left\{ \left\Vert T\right\Vert \left\Vert T^{-1}\right\Vert :T\in\mathcal{L}\left(X,Y\right)\text{ and }T^{-1}\in\mathcal{L}\left(Y,X\right)\right\}
%\] and
between $0$-symmetric convex bodies in a Euclidean space $\mathbb{E}$  is defined as:
\[
\delta^{BM}\left(P,Q\right):=\inf\left\{ \lambda\geq1:T:\mathbb{E}\rightarrow\mathbb{E}\text{ is a bijective linear map and }Q\subseteq TP\subseteq\lambda Q\right\}.
\]
A complete exposition of the Banah-Mazur distance can be found in \cite{Tomczak-Jaegermann1989}.

\begin{prop}
\label{prop:continuity wrt BM distance}
Let $P_{i},Q_{i}\subset\mathbb{E}_{i}$, $i=1,...,l$, be $0$-symmetric convex bodies.
If $\otimes_{\alpha}$ is a tensor product convex bodies, then
\[
\delta^{BM}\left(P_{1}\otimes_{\alpha}\cdots\otimes_{\alpha}P_{l},Q_{1}\otimes_{\alpha}\cdots\otimes_{\alpha}Q_{l}\right)\leq\delta^{BM}\left(P_{1},Q_{1}\right)\cdots\delta^{BM}\left(P_{l},Q_{l}\right).
\]
\end{prop}
\begin{proof}
Let us fix   $i\in\left\{ 1,...,l\right\}$   and let $\lambda\geq\delta^{BM}\left(P_{i},Q_{i}\right)$.  Then, there exists  a linear isomorphism $T_{i}:\mathbb{E}_{i}\rightarrow\mathbb{E}_{i}$ such that
$Q_{i}\subseteq T_{i}\left(P_{i}\right)\subseteq\lambda Q_{i}$.
 By Proposition \ref{prop:linear invariance of products}, we have $P_{1}\otimes_{\alpha}\cdots\otimes_{\alpha}T_{i}P_{i}
 \otimes_{\alpha}\cdots\otimes_{\alpha}P_{l}=I_{\mathbb{E}_{1}}
 \otimes\cdots\otimes T_{i}\otimes\cdots\otimes I_{\mathbb{E}_{l}}\left(P_{1}\otimes_{\alpha}\cdots
 \otimes_{\alpha}P_{i}\otimes_{\alpha}\cdots
 \otimes_{\alpha}P_{l}\right).$ Therefore, if $S=I_{\mathbb{E}_{1}}\otimes\cdots\otimes T_{i}\otimes\cdots\otimes I_{\mathbb{E}_{l}}$ then
 \begin{multline*}
P_{1}\otimes_{\alpha}\cdots\otimes_{\alpha}Q_{i}\otimes_{\alpha}
\cdots\otimes_{\alpha}P_{l} \subseteq S\left(P_{1}\otimes_{\alpha}\cdots\otimes_{\alpha}P_{i}
\otimes_{\alpha}\cdots\otimes_{\alpha}P_{l}\right)\\
 \subseteq P_{1}\otimes_{\alpha}\cdots\otimes_{\alpha}\lambda Q_{i}\otimes_{\alpha}\cdots\otimes_{\alpha}P_{l}
 =\lambda\left(P_{1}\otimes_{\alpha}\cdots\otimes_{\alpha}Q_{i}\otimes_{\alpha}\cdots\otimes_{\alpha}P_{l}\right).
\end{multline*}
Consequently,
  \begin{equation*}
\delta^{BM}\left(P_{1}\otimes_{\alpha}\cdots\otimes_{\alpha}P_{i}\otimes_{\alpha}\cdots\otimes_{\alpha}P_{l},P_{1}\otimes_{\alpha}\cdots\otimes_{\alpha}Q_{i}\otimes_{\alpha}\cdots\otimes_{\alpha}P_{l}\right)\leq\delta^{BM}\left(P_{i},Q_{i}\right).\label{eq:cota banac mazur proposicion}
\end{equation*}
The result follows by iterating  this relation  along with  the multiplicative triangle inequality of $\delta^{BM}$.
\end{proof}

%Here and subsequently $X$, $Y$ or $X_i$ will  denote Banach spaces.  The closed unit ball of $X$ will be  denoted by  $B_{X}$ and  its  dual space by $X^{*}.$ We write $\mathcal{L}\left(X,Y\right)$ to denote the space of  bounded linear operators from $X$ to $Y.$ If $T\in\mathcal{L}\left(X,Y\right),$ $\left\Vert T\right\Vert$ denotes its usual norm.

%----Esto ya estaba desaparecido por Luisa
%%A norm $\alpha\left(\cdot\right)$ on $\otimes_{i=1}^{l}X_{i}$ is a \textbf{reasonable crossnorm} if and only if
%%\begin{equation}
%%\label{eq:Caracterizacion de normars razonables cruzadas}
%%\epsilon\left(u\right)\leq\alpha\left(u\right)\leq\pi\left(u\right)\text{ for every }u\in\otimes_{i=1}^{l}X_{i}.
%%\end{equation}
%-------

\subsection{Bijection with tensor norms}\label{subs: Tensor norms}

Recall that
a  \textbf{tensor norm $\alpha$} of order\textbf{ $l$} on the class of  finite dimensional  (abbreviated  f.d.) Banach spaces  assigns to each $l$-tuple
$M_{1},\ldots,M_{l}$ of f.d. Banach spaces a norm $\alpha_{M_{1},\ldots,M_{l}}$
on the tensor product $\otimes_{i=1}^{l}M_{i}$  (simply denoted $\alpha$, and  called a \textbf{tensor norm}) such that the two
following conditions are satisfied:
\begin{enumerate}
\item $\alpha$ is a \textbf{reasonable crossnorm}.
%, i.e. for every $u\in\otimes_{i=1}^{l}M_{i}$
%\[
%%\epsilon\left(u\right)\leq\alpha\left(u\right)\leq\pi\left(u\right),
%\]
\item $\alpha$ satisfies the \textbf{uniform property} i.e. for each $T_{i}\in\mathcal{L}\left(M_{i};N_{i}\right)$ $i=1,\ldots,l,$
\[
\left\Vert T_{1}\otimes\cdots\otimes T_{l}:\left(\otimes_{i=1}^{l}M_{i},\alpha\right)
\rightarrow\left(\otimes_{i=1}^{l}N_{i},\alpha\right)\right\Vert \leq\left\Vert T_{1}\right\Vert \cdots\left\Vert T_{l}\right\Vert .
\]
\end{enumerate}
In relation with the so called {\sl  Minkowski functional}, it was seen in \cite[Corollary 3.4]{tensorialbodies} that  a   {\sl reasonable crossnorm}   corresponds to  a    {\sl tensorial body} (see  \cite[Definition 3.3]{tensorialbodies}) and vice versa.

Given a tensor norm  $\alpha\left(\cdot\right)$  on $\otimes_{i=1}^{l}M_{i}$,
$\otimes_{\alpha,i=1}^{l}M_{i}$ will denote the normed  space $\left(\otimes_{i=1}^{l}M_{i},\alpha\right)$.

  The dual tensor norm  of $\alpha$, $\alpha^{\prime}$, is defined as follows: given any f.d. normed spaces   $M_{1},\ldots,M_{l}$, for  every $u\in\otimes_{i=1}^{l}M_{i},$
\begin{equation}\label{eq: defn dual norm}
\alpha_{M_{1},...,M_{l}}^{\prime}\left(u\right):=\sup\left\{ \left|\varphi\left(u\right)\right|:\alpha_{M_{1}^{*},
\ldots,M_{l}^{*}}\left(\varphi\right)\leq1\right\}.
\end{equation}

A  tensor norm $\alpha$ is called {\bf injective} if for each subspace $E_i$ of a Banach space $X_i,$ $\otimes_{\alpha,i=1}^{l}E_i$ is a subspace of $\otimes_{\alpha,i=1}^{l}X_i.$ A tensor norm $\alpha$ is called {\bf projective} is for every $l$-tuple of  quotient operators $S_i:X_i\rightarrow Y_i$, $S_1\otimes\cdots\otimes S_l:\otimes_{\alpha,i=1}^{l}X_i\rightarrow \otimes_{\alpha,i=1}^{l}Y_i$ is a quotient operator.

For a detailed exposition  about tensor norms on Banach spaces, we suggest the reader the monographs \cite{defantfloret, diestelfourie, Ryan2013}.

\begin{prop}
\label{prop:independence of scalar product}
Let $\otimes_{\alpha}$ be a tensor product of $0$-symmetric convex
bodies.  For every $l$-tuple $M_{i},$  $i=1,...,l,$ of f.d. Banach spaces, consider
\begin{equation}\label{eq: tensor norm given convex}
 \left\Vert z\right\Vert _{\otimes_{\alpha}}:=g_{B_{1}\otimes_{\alpha}\cdots\otimes_{\alpha}B_{l}}\left(z\right)\text{ for }z\in\otimes_{i=1}^{l}M_{i}
\end{equation}
where $B_{i}$  is the closed unit  ball of $M_{i}$.  Then, $\left\Vert \cdot\right\Vert _{\otimes_{\alpha}}$ defines a tensor norm  on the class of finite dimensional Banach spaces.
\end{prop}
\begin{proof}
Let $M_{i},$ $i=1,\ldots,l,$ be f.d.  normed spaces and let fix  a scalar product on each of them.  By the discussion in Subsection \ref{subs: analogue for 0-symm} and
(\ref{eq: p. inj and p.proj  unit ball inj. norm and proj norm}), we have that (\ref{eq: tensor norm given convex}) satisfies condition (1) above. That is, it is  a reasonable crossnorm  on $\otimes_{i=1}^{l}M_{i}$.

 For $\left\Vert \cdot\right\Vert _{\otimes_{\alpha}}$ to be well defined on $\otimes_{i=1}^{l}M_{i}$,  we must check that it does not depend on the scalar product we have considered on each $M_i$. To that end, let  $\left[\cdot,\cdot\right]_{i},\left\langle \cdot,\cdot\right\rangle _{i}$
be scalar products on $M_{i}$, let $T_{i}:\left(M_{i},\left[\cdot,\cdot\right]_{i}\right)\rightarrow\left(M_{i},\left\langle \cdot,\cdot\right\rangle _{i}\right)$ denote
the identity map and let $Q_{i}:=T_{i}B_{M_{i}}.$ From   Proposition \ref{prop:linear invariance of products},
\begin{align*}
B_{M_{1}}\otimes_{\alpha}\cdots\otimes_{\alpha}B_{M_{l}}&  = T_{1}\otimes\cdots\otimes T_{l}\left(B_{M_{1}}\otimes_{\alpha}\cdots\otimes_{\alpha}
B_{M_{l}}\right)\\
 & = T_{1}B_{M_{1}}\otimes_{\alpha}\cdots\otimes_{\alpha}T_{l}B_{M_{l}}
  =Q_{1}\otimes_{\alpha}\cdots\otimes_{\alpha}Q_{l}.
\end{align*}

Therefore, $g_{B_{M_{1}}\otimes_{\alpha}\cdots\otimes_{\alpha}B_{M_{l}}}
=g_{Q_{1}\otimes_{\alpha}\cdots\otimes_{\alpha}Q{}_{l}}.$
 Consequently, $\left\Vert \cdot\right\Vert _{\otimes_{\alpha}}$ is a well defined reasonable crossnorm on the class of finite dimensional normed spaces.

To finish, we have to prove that this norm satisfies the uniform property  for tensor norms.  To see this, take $T_{i}\in\mathcal{L}\left(M_{i},N_{i}\right)$
such that $\left\Vert T_{i}\right\Vert \leq1.$ Then, $T_{i}\left(B_{i}\right)\subseteq B_{N_{i}}.$
By the uniform property of $\otimes_{\alpha},$ we have
\[
T_{1}\otimes\cdots\otimes T_{l}\left(B_{1}\otimes_{\alpha}\cdots\otimes_{\alpha}B_{l}\right)\subseteq B_{N_{1}}\otimes_{\alpha}\cdots\otimes_{\alpha}B_{N_{l}}.
\]
This implies that
$
T_{1}\otimes\cdots\otimes T_{l}:\left(\otimes_{i=1}^{l}M_{i},\left\Vert \cdot\right\Vert _{\otimes_{\alpha}}\right)\rightarrow\left(\otimes_{i=1}^{l}N_{i},\left\Vert \cdot\right\Vert _{\otimes_{\alpha}}\right)
$
 has norm $\leq 1$. So, $\left\Vert \cdot\right\Vert _{\otimes_{\alpha}}$
is uniform, as required.
\end{proof}

\begin{prop}
\label{prop:tensor norms determine tensor sym cb}
Let $\|\cdot\|_{\alpha}$ be a tensor norm on f.d. spaces.  For every $l$-tuple $P_{i}\subset\mathbb{E}_{i}$ $i=1,...,l,$  of $0$-symmetric
convex bodies consider
\begin{equation}\label{eq: convex product given tensor norm}
P_{1}\otimes_{{\alpha}}\cdots\otimes_{{\alpha}}P_{l}:=
B_{\otimes_{\|\cdot\|_{\alpha},i=1}^l
\left(\mathbb{\mathbb{E}}_{i},g_{P_{i}}\right)}.
\end{equation}
Then, $\otimes_{{\alpha}}$ defines a  tensor
product of $0$-symmetric convex bodies.
\end{prop}

\begin{proof}
%The proof is a consequence of conditions (1) and (2) in the definition of tensor norm.
Let $P_{i}\subset\mathbb{E}_{i}$ be $0$-symmetric convex bodies. Denote by $E_{i}=\left(\mathbb{\mathbb{E}}_{i},g_{P_{i}}\right)$
the normed space whose  closed unit ball is $P_i$.
Since $\|\cdot\|_{\alpha}$ is a tensor norm, then
$B_{\otimes^l_{\|\cdot\|_{\alpha},i=1}\left(\mathbb{\mathbb{E}}_{i},g_{P_{i}}\right)}$ is a $0$-symmetric convex body for which
$
\epsilon\left(z\right)\leq\alpha\left(z\right)\leq\pi\left(z\right),\text{ for }z\in\otimes_{i=1}^{l}E_{i}.
$

Thus, from  (\ref{eq: p. inj and p.proj  unit ball inj. norm and proj norm}) and (\ref{eq: convex product given tensor norm}), we have that
$
P_{1}\otimes_{\pi}\cdots\otimes_{\pi}P_{l}\subseteq P_{1}\otimes_{\|\cdot\|_{\alpha}}\cdots\otimes_{\|\cdot\|_{\alpha}}P_{l}\subseteq P_{1}\otimes_{\epsilon}\cdots\otimes_{\epsilon}P_{l}.
$
The uniform property of $\otimes_{{\alpha}}$ follows directly from the uniform property of the norm $\alpha.$
\end{proof}

\begin{thm}\label{thm: main theorem the bijection}  Let $ \otimes_{\alpha}$ and $  \|\cdot\|_{\otimes_{\alpha}}$
be  as in Propositions  \ref{prop:independence of scalar product} and \ref{prop:tensor norms determine tensor sym cb}.
   Then, the  mapping
\[
\begin{array}{rcl}
  \varPsi: \left\{\text{\stackanchor{Tensor product of}{$0$-symm convex bodies}} \right\}& \longrightarrow & \left\{\text{\stackanchor{Tensor norms on}{f.d. Banach spaces}} \right\} \\
   & & \\
  \otimes_{\alpha} & \mapsto & \|\cdot\|_{\otimes_\alpha}.
\end{array}
\]
is a bijection.  It holds that
\begin{enumerate}
  \item[i)]  $\otimes_{\alpha^{\prime}}$ is dual to   $\otimes_{\alpha}$ if and only if   $\|\cdot\|_{\otimes_{\alpha^{\prime}}}$ is dual to  $\|\cdot\|_{\otimes_\alpha}$.
  \item[ii)] $\otimes_{\alpha}$  is injective (projective) if and only if  $\|\cdot\|_{\otimes_\alpha}$ is injective (resp. projective).

\end{enumerate}
\end{thm}
\begin{proof} By Propositions \ref{prop:independence of scalar product} and \ref{prop:tensor norms determine tensor sym cb}, to prove that $\varPsi$ is a bijection it is  only necessary  to verify that, whenever $\beta\left(\cdot\right)=\left\Vert \cdot\right\Vert _{\otimes_{\alpha}},$
then $\otimes_{\alpha}=\otimes_{\beta}.$ To that end, consider an  $l$-tuple  $P_{i}\subseteq\mathbb{E}_{i},$ $i=1,...,l,$ of $0$-symmetric convex bodies. Then, applying  (\ref{eq: convex product given tensor norm}) in the first equality and  (\ref{eq: tensor norm given convex}) in the third equality, we have
\begin{align*}
P_{1}\otimes_{\beta}\cdots\otimes_{\beta}P_{l} & :=\left\{ z\in\otimes_{i=1}^{l}\left(\mathbb{\mathbb{E}}_{i},g_{P_{i}}\right)
:\beta\left(z\right)\leq1\right\} \\
 & =\left\{ z\in\otimes_{i=1}^{l}\left(\mathbb{\mathbb{E}}_{i},g_{P_{i}}\right)
 :g_{P_{1}\otimes_{\alpha}\cdots\otimes_{\alpha}P_{l}}\left(z\right)
 \leq1\right\}
  =: P_{1}\otimes_{\alpha}\cdots\otimes_{\alpha}P_{l}.
\end{align*}
In that case,  if $\left\Vert \cdot\right\Vert _{\alpha}$ is
a tensor norm on finite dimensional normed spaces and  $\otimes_{\beta}=\otimes_{\left\Vert \cdot\right\Vert _{\alpha}},$
then $\left\Vert \cdot\right\Vert _{\otimes_{\beta}}=\left\Vert \cdot\right\Vert _{\alpha}.$

To prove $i)$, let $\otimes_{\alpha}$  be     a tensor product of $0$-symmetric convex bodies and let  $\beta\left(\cdot\right):=\left\Vert \cdot\right\Vert _{\otimes_{\alpha}}$.   To see that
$\beta^{\prime}\left(\cdot\right)=\left\Vert \cdot\right\Vert _{\otimes_{\alpha^{\prime}}}$, we have to check that given  $l$  finite dimensional normed spaces  $M_{i}$,
\[
\beta^{\prime}\left(z\right)=g_{B_{M_{1}}
\otimes_{\alpha^{\prime}}\cdots\otimes_{\alpha^{\prime}}B_{M_{l}}}
\left(z\right)\text{ for every }z\in\otimes_{i=1}^{l}M_{i}.
\]
To that end, let us  fix  a scalar product
   $\left\langle \cdot,\cdot\right\rangle _{i}$ on each $M_i$ (we checked in Proposition \ref{prop:independence of scalar product} the independence on  the selected scalar products).

Let $T_{i}:M_{i}\rightarrow M_{i}^{*}$ be defined as  $T_{i}(x):=\left\langle \cdot,x\right\rangle _{i}.$ Then, $T_{i}\left(B_{M_{i}}^{\circ}\right)=B_{M_{i}^{*}},$ and
$T_{1}\otimes\cdots\otimes T_{l}:\otimes_{i=1}^{l}M_{i}
\rightarrow\otimes_{i=1}^{l}M_{i}^{*}$  is the map sending $w\mapsto\left\langle \cdot,w\right\rangle _{H}.$
Thus, by Proposition \ref{prop:linear invariance of products}, we have:
\begin{align*}
T_{1}\otimes\cdots\otimes T_{l}\left(B_{M_{1}}^{\circ}\otimes_{\alpha}\cdots\otimes_{\alpha}B_{M_{l}}^{\circ}\right) & =T_{1}\left(B_{M_{1}}^{\circ}\right)\otimes_{\alpha}\cdots\otimes_{\alpha}T_{l}\left(B_{M_{l}}^{\circ}\right)\\
 & =B_{M_{1}^{*}}\otimes_{\alpha}\cdots\otimes_{\alpha}B_{M_{l}^{*}}.
\end{align*}
Then, for every
$\varphi\in\otimes_{i=1}^{l}M_{i}^{*}$,  it holds that $g_{B_{M_{1}^{*}}\otimes_{\alpha}\cdots
\otimes_{\alpha}B_{M_{l}^{*}}}\left(\varphi\right)\leq 1$ if and only if
$
T_{1}^{-1}\otimes\cdots\otimes T_{l}^{-1}\left(\varphi\right)\in B_{M_{1}}^{\circ}\otimes_{\alpha}\cdots\otimes_{\alpha}B_{M_{l}}^{\circ}.
$
Since
$
\left\langle z,T_{1}^{-1}\otimes\cdots\otimes T_{l}^{-1}\left(\varphi\right)\right\rangle _{H}=\varphi\left(z\right),
$ we obtain
\begin{multline*}
  \beta^{\prime}\left(z\right):=\sup\left\{ \left|\varphi\left(z\right)\right|:g_{B_{M_{1}^{*}}
\otimes_{\alpha}\cdots\otimes_{\alpha}B_{M_{l}^{*}}}
\left(\varphi\right)\leq1\right\}=\\
 \sup\left\{ \left|\langle z, w\rangle_H\right|:g_{B_{M_{1}^{\circ}}
\otimes_{\alpha}\cdots\otimes_{\alpha}B_{M_{l}^{\circ}}}
\left(w\right)\leq1\right\} =\\
g_{({B_{M_{1}^{\circ}}
\otimes_{\alpha}\cdots\otimes_{\alpha}B_{M_{l}^{\circ}}})^{\circ}}(z)=
g_{B_{M_{1}}
\otimes_{\alpha^{\prime}}\cdots\otimes_{\alpha^{\prime}}B_{M_{l}}}
\left(z\right).
\end{multline*}
We have already proved that $\varPsi$ preserves duality. With this  and the bijectivity of $\varPsi$, it follows that
 $\varPsi^{-1}$ preserves duality, too.

To prove $ii)$ let us suppose that $\otimes_{\alpha}$ is injective. Consider  any finite dimensional normed spaces $M_{i}$ and any fixed scalar product on $M_{i}$, $i=1,...,l$,  $\left\langle \cdot,\cdot\right\rangle _{i}$.

 We consider, on every subspace $N_{i}\subseteq M_{i}$,
 the scalar product dermined by the restriction of $\left\langle \cdot,\cdot\right\rangle _{i}$
to $M_{i}.$
The injectivity of  $\otimes_{\alpha}$ implies
\[
\left(B_{M_{1}}\cap N_{1}\right)\otimes_{\alpha}\cdots\otimes_{\alpha}\left(B_{M_{l}}\cap N_{l}\right)=B_{M_{1}}\otimes_{\alpha}\cdots
\otimes_{\alpha}B_{M_{l}}\cap\otimes_{i=1}^{l}N_{i}.
\]
Since  $B_{M_{i}}\cap N_{i}=B_{N_{i}}$, we get that
\[
g_{B_{N_{1}}\otimes_{\alpha}\cdots\otimes_{\alpha}B_{N_{l}}}\left(z\right)=g_{B_{M_{1}}\otimes_{\alpha}\cdots\otimes_{\alpha}B_{M_{l}}}\left(z\right)\text{ for }z\in\otimes_{i=1}^{l}N_{i}.
\]
This already shows that
$\left\Vert \cdot\right\Vert _{\otimes_{\alpha}}$ is injective. Since all the steps of the proof are reversible, it follows that the reciprocal is also true.

The projective case follows using $i)$, namely,  that the bijection $\varPsi$  preserves duality, along with the dual relation between injective and projective tensor products (Proposition \ref{prop:property dual product} and \cite[Proposition 7.5]{Ryan2013}).
\end{proof}

Given a tensor norm $\alpha$ on normed spaces, its injective associate tensor norm  $/\alpha\setminus$ is the biggest injective tensor norm $/\alpha\setminus\leq\alpha$. Similarly, its projective associate tensor norm $\setminus\alpha/$ is the smallest projective tensor norm $\setminus\alpha/\geq\alpha.$ See \cite{defantfloret,Ryan2013} for a deeper discussion about these tensor norms. With this notation, we have:

\begin{prop}
\label{prop: injectiva y projectiva a partir del teorema principal}
Let $\otimes_\alpha$ be a tensor product of $0$-symmetric convex bodies and let $\left\Vert\cdot\right\Vert _{\otimes_{\alpha}}$ be its corresponding tensor norm. Then, the tensor norms on finite dimensions that correspond (under the bijection of Theorem \ref{thm: main theorem the bijection}) to the products $\otimes_{\alpha^{inj}},$ $\otimes_{\alpha^{proj}}$ are the injective and projective associated to $\left\Vert\cdot\right\Vert _{\otimes_{\alpha}}$ respectively.
\end{prop}

\begin{proof}
From (ii) of Theorem \ref{thm: main theorem the bijection}, it follows that $\left\Vert\cdot\right\Vert _{\otimes_{\alpha^{inj}}}$  is an injective tensor norm such that $\left\Vert\cdot\right\Vert _{\otimes_{\alpha^{inj}}}\leq\left\Vert\cdot\right\Vert _{\otimes_{\alpha}}.$ Thus, by the definition of $\otimes_{\alpha^{inj}}$ and Theorem \ref{thm: main theorem the bijection}, it has to be the biggest one below $\left\Vert\cdot\right\Vert _{\otimes_{\alpha}}.$ The result for the projective associate tensor norm and $\left\Vert\cdot\right\Vert _{\otimes_{\alpha^{proj}}}$ follows from the previous one, the definition of $\otimes_{\alpha^{proj}}$ and the duality exposed in (i) of Theorem \ref{thm: main theorem the bijection}.
\end{proof}

\subsection{Explicit representation of the  biggest injective and  the smallest projective tensor products. }\label{subs: explicit hulls}
%Below we exhibit explicit representations of the biggest injective tensor product $\otimes_\pi\setminus$ and the smallest projective tensor product $\otimes_\epsilon/$.
%To that end, we will make visible the relation between the injective and projective associate tensor norms and tensor products of $0$-symmetric convex bodies.

Every separable Banach space can be isometrically embedded into $\ell_{\infty}$ and is a quotient of $\ell_1$  (\cite[pp. 19]{JOHNSON20011} and  \cite[pp. 17]{JOHNSON20011}). With  these properties, we will obtain   explicit representations of $\otimes_{\pi^{inj}}$ and $\otimes_{\epsilon^{proj}}$, respectively.

\begin{lem}
\label{lem: indep del enaje linfity}
For each $i=1,\ldots,l$, the following statements hold:
\begin{enumerate}
\item Let  $j_i:(\mathbb{E}_{i}, g_{P_i})\hookrightarrow \ell_\infty$ and $k_i:(\mathbb{E}_{i}, g_{P_i})\hookrightarrow \ell_\infty$ be isometric embeddings. Then,
$$
(\otimes_{i=1}^l k_i)^{-1}(B_{\ell_{\infty}\hat{\otimes}_{\pi}\cdots \hat{\otimes}_{\pi} \ell_{\infty}}\cap\otimes_{i=1}^l {k_i(\mathbb{E}_{i})})=(\otimes_{i=1}^l j_i)^{-1} (B_{\ell_{\infty}\hat{\otimes}_{\pi}\cdots \hat{\otimes}_{\pi} \ell_{\infty}}\cap\otimes_{i=1}^lj_i(\mathbb{E}_{i})).
$$
\item Let $q_i:\ell_1\rightarrow(\mathbb{E}_{i}, g_{P_i})$ and $h_i:\ell_1\rightarrow(\mathbb{E}_{i}, g_{P_i})$ be quotient maps. Then,
\[
(q_1\otimes \cdots\otimes q_l)(B_{\ell_1\hat{\otimes}_{\epsilon}\cdots \hat{\otimes}_{\epsilon} \ell_{1}})=(h_1\otimes \cdots\otimes h_l)(B_{\ell_1\hat{\otimes}_{\epsilon}\cdots \hat{\otimes}_{\epsilon} \ell_{1}})
\]
\end{enumerate}
\end{lem}

\begin{proof} (1). We will use  the injectivity of the space $\ell_{\infty}$ (\cite[Ex. 1.7]{defantfloret}) to extend  linear  mappings defined on subspaces,   to  linear mappings preserving the norms.
The isometries $j_i\circ (k_i)^{-1}:k_i(\mathbb{E}_{i})\rightarrow j_i(\mathbb{E}_{i})\subset \ell_{\infty}$ can, thus, be extended to
 norm-one mappings $\mu_i:\ell_{\infty}\rightarrow \ell_{\infty}$.  The restriction of $\otimes_{i=1}^n\mu_i:{\ell_{\infty}\hat{\otimes}_{\pi}\cdots \hat{\otimes}_{\pi} \ell_{\infty}}\rightarrow {\ell_{\infty}\hat{\otimes}_{\pi}\cdots \hat{\otimes}_{\pi} \ell_{\infty}}$ satisfies
 $$(\otimes_{i=1}^l (j_i\circ (k_i)^{-1})(B_{\ell_{\infty}\hat{\otimes}_{\pi}\cdots \hat{\otimes}_{\pi} \ell_{\infty}}\cap\otimes_{i=1}^l {k_i(\mathbb{E}_{i})})\subseteq B_{\ell_{\infty}\hat{\otimes}_{\pi}\cdots \hat{\otimes}_{\pi} \ell_{\infty}}\cap\otimes_{i=1}^lj_i(\mathbb{E}_{i}).
  $$
   The same argument used with the isometries $k_i\circ (j_i)^{-1}$ says that this contention is, indeed, an equality:
   $$(\otimes_{i=1}^l j_i)((\otimes_{i=1}^l k_i)^{-1}(B_{\ell_{\infty}\hat{\otimes}_{\pi}\cdots \hat{\otimes}_{\pi} \ell_{\infty}}\cap\otimes_{i=1}^l {k_i(\mathbb{E}_{i})}))= B_{\ell_{\infty}\hat{\otimes}_{\pi}\cdots \hat{\otimes}_{\pi} \ell_{\infty}}\cap\otimes_{i=1}^lj_i(\mathbb{E}_{i}).
  $$
Hence, applying $(j_1\otimes\cdots\otimes j_l)^{-1}$ to the last equality, we obtain the desired result.

(2). This part follows from the lifting property of $\ell_1$ (\cite[3.12]{defantfloret}). Indeed, let $\rho>0$ and $\tilde{q}_i:\ell_1\rightarrow\ell_1$ be such that $q_i=h_i\circ\tilde{q}_i$ and $\Vert\tilde{q}_i\Vert\leq(1+\rho)\Vert q_i\Vert.$ Then, since each $q_i$ has norm one, we have that $\tilde{q}_1\otimes\cdots\otimes\tilde{q}_i(B_{\ell_1\hat{\otimes}_{\epsilon}\cdots \hat{\otimes}_{\epsilon} \ell_{1}})\subseteq(1+\rho)^l(B_{\ell_1\hat{\otimes}_{\epsilon}\cdots \hat{\otimes}_{\epsilon} \ell_{1}}).$ Applying $h_1\otimes\cdots\otimes h_l$  to the previous inclusion, we have:
\[
q_1\otimes\cdots\otimes q_l(B_{\ell_1\hat{\otimes}_{\epsilon}\cdots \hat{\otimes}_{\epsilon} \ell_{1}})\subseteq(1+\rho)^{l}h_1\otimes\cdots\otimes h_l(B_{\ell_1\hat{\otimes}_{\epsilon}\cdots \hat{\otimes}_{\epsilon} \ell_{1}}).
\]
Thus, $q_1\otimes\cdots\otimes q_l(B_{\ell_1\hat{\otimes}_{\epsilon}\cdots \hat{\otimes}_{\epsilon} \ell_{1}})\subseteq h_1\otimes\cdots\otimes h_l(B_{\ell_1\hat{\otimes}_{\epsilon}\cdots \hat{\otimes}_{\epsilon} \ell_{1}}).$

The lifting property used for the quotient maps $h_i$ shows the other inclusion. This yields to the desired equality.
\end{proof}

\begin{thm}
\label{thm: la menor injectiva}
For    $i=1,...,l$, let  $P_{i}\subset\mathbb{E}_{i}$ be  $0$-symmetric convex bodies.
\begin{enumerate}
  \item  For any isometric embeddings
$j_i:(\mathbb{E}_{i}, g_{P_i})\hookrightarrow \ell_\infty,$
\begin{equation*}\label{eq: smaller injective}
P_{1}\otimes_{\pi^{inj}}\cdots
\otimes_{\pi^{inj}}P_{l}=\left(j_1\otimes \cdots\otimes j_l\right)^{-1}\left(B_{\ell_{\infty}\hat{\otimes}_{\pi}\cdots \hat{\otimes}_{\pi} \ell_{\infty}}\cap\otimes_{i=1}^{l}j_i(\mathbb{\mathbb{E}}_{i})\right).
\end{equation*}
 \item For    any  quotient mappings
$q_i:\ell_1\rightarrow (\mathbb{E}_{i}, g_{P_i}),$
\begin{equation*}\label{eq: smaller injective}
P_{1}\otimes_{\epsilon^{proj}}\cdots
\otimes_{\epsilon^{proj}}P_{l}=(q_1\otimes \cdots\otimes q_l)(B_{\ell_1\hat{\otimes}_{\epsilon}\cdots \hat{\otimes}_{\epsilon} \ell_{1}}).
\end{equation*}
\end{enumerate}
\end{thm}

\begin{proof}
(1). Since $\otimes_\pi$ corresponds, under the bijection given in Theorem \ref{thm: main theorem the bijection}, to the projective norm  $\pi(\cdot),$ then, by Proposition \ref{prop: injectiva y projectiva a partir del teorema principal}, $\otimes_{\pi^{inj}}$ corresponds to the biggest injective tensor norm $/\pi\setminus.$ Hence, for each tuple $P_i\subset\mathbb{E}_i,$ $i=1,\ldots,l,$
$P_1\otimes_{\pi^{inj}}\cdots\otimes_{\pi^{inj}}P_l=B_{\otimes_{/\pi\setminus,i=1}^l\left(\mathbb{E}_i,g_{P_i}\right)}.$

If $j_i:\left(\mathbb{E}_i,g_{P_i}\right)\hookrightarrow\ell_\infty$ are are isometric emmbedings as in  \cite[20.7]{defantfloret}, then
$/\pi\setminus(u)=\pi(j_1\otimes\cdots\otimes j_l(u)),$
for every $u\in\otimes_{i=1}^l\left(\mathbb{E}_i,g_{P_i}\right).$ So,
\[
P_{1}\otimes_{\pi^{inj}}\cdots
\otimes_{\pi^{inj}}P_{l}=(j_1\otimes \cdots\otimes j_l)^{-1}\big\{(B_{\ell_{\infty}\hat{\otimes}_{\pi}\cdots \hat{\otimes}_{\pi} \ell_{\infty}})\cap (j_1\otimes \cdots\otimes j_l)(\otimes_{i=1}^{l}\mathbb{\mathbb{E}}_{i})\big\}.
\]
By Lemma \ref{lem: indep del enaje linfity}, the latter does not depend  on the election of the embeddings $j_i$.

To prove $(2),$ we use that every separable Banach space is a quotient of $\ell_1.$  Also, by \cite[20.6]{defantfloret}, the smallest projective tensor norm $\setminus\epsilon/$ is such that $B_{\otimes_{\setminus\epsilon/}\left(\mathbb{E}_i,g_{P_i}\right)}=(q_1\otimes \cdots\otimes q_l)\left(B_{\ell_1\hat{\otimes}_{\epsilon}\cdots \hat{\otimes}_{\epsilon} \ell_{1}}\right),$ where $q_i:\ell_1\rightarrow (\mathbb{E}_{i}, g_{P_i})$ are the quotient mappings of \cite[20.6]{defantfloret}.  Thus, by applying Proposition \ref{prop: injectiva y projectiva a partir del teorema principal} and \ref{lem: indep del enaje linfity}, we obtain the desired result.
%let us  consider quotient mappings $q_i$ as in the hypothesis.   For each $i=1\ldots,l$, the mappings $q_i^*:(\mathbb{E}_i,g_{P^o_i})\hookrightarrow \ell_{\infty}$ are isometric embeddings. By $(1)$ we have that
%\begin{multline*}
%P_{1}^o\otimes_{\pi\setminus}\cdots
%\otimes_{\pi\setminus}P_{l}^o=(q^*_1\otimes \cdots\otimes q^*_l)^{-1}\big\{(B_{\ell_{\infty}\hat{\otimes}_{\pi}\cdots \hat{\otimes}_{\pi} \ell_{\infty}})\cap (q^*_1\otimes \cdots\otimes q^*_l)(\otimes_{i=1}^{l}\mathbb{\mathbb{E}}_{i})\big\}.
%\end{multline*} Now, we use  the relation  $(T(P))^o=(T^t)^{-1}(P^o)$ ( {\color{red}{buscar otra referencia,  Bourbaki TVS II.47 Proposition 6}}.) with $T^t=q_1^*\otimes\cdots \otimes q_l^*$ and $P^o=B_{\ell_{\infty}\hat{\otimes}_{\pi}\cdots \hat{\otimes}_{\pi} \ell_{\infty}}\cap (q^*_1\otimes \cdots\otimes q^*_l)(\otimes_{i=1}^{l}\mathbb{\mathbb{E}}_{i})$. Then, the duality between the projective  and the injective tensor norms gives rise to the equality:
%\begin{multline*}
%P_{1}\otimes_{\epsilon/}\cdots
%\otimes_{\epsilon/}P_{l}=
%(P_{1}^o\otimes_{\pi\setminus}\cdots
%\otimes_{\pi\setminus}P_{l}^o)^o=(q_1\otimes \cdots\otimes q_l)(B_{\ell_{1}\hat{\otimes}_{\epsilon}\cdots \hat{\otimes}_{\epsilon} \ell_{1}}).
%\end{multline*}
%{\color{red}{EL argumento seria: $F:=\otimes \mathbb{E}_i$ como subespacio cerrado de $M:=\otimes_{\pi}\ell_{\infty}$ tiene como dual el cociente $M^*/ F^{\perp}$ pero por ser de dim finita F, acaba siendo tambien  el cociente del prepredual $(\otimes_{\epsilon}\ell_1)/ ^{\perp}F$}}. Revisalo, por favor.
\end{proof}

\begin{cor}
\label{coro: en linfty coinciden}
For $d_i\in \mathbb{N}$, $i=1,\ldots, l$,
\begin{equation*}\label{eq: en linfty coinciden}
B_{\infty}^{d_1}{\otimes}_{\pi^{inj}}\cdots {\otimes}_{\pi^{inj}} B_{\infty}^{d_l}=B_{\infty}^{d_1}{\otimes}_{\pi}\cdots {\otimes}_{\pi} B_{\infty}^{d_l}
\end{equation*}
\begin{equation*}\label{eq: en l1 coinciden}
B_{1}^{d_1}{\otimes}_{\epsilon^{proj}}\cdots {\otimes}_{\epsilon^{proj}} B_{1}^{d_l}=B_{1}^{d_1}{\otimes}_{\epsilon}\cdots {\otimes}_{\epsilon} B_{1}^{d_l}.
\end{equation*}
\end{cor}

\begin{proof}
Let us denote by
$j_i:\ell_{\infty}^{d_i}\hookrightarrow \ell_{\infty}$
the inclusion mappings into the first $d_i$ coordinates and $\Pi_i:\ell_{\infty}\rightarrow \ell_{\infty}^{d_i}$ the projection mappings.  Since $\pi$ is a projective tensor norm, we have that $$(\otimes_{i=1}^l \Pi_i):{\ell_{\infty}\hat{\otimes}_{\pi}\cdots \hat{\otimes}_{\pi} \ell_{\infty}}\rightarrow {\ell_{\infty}^{d_1}\hat{\otimes}_{\pi}\cdots \hat{\otimes}_{\pi} \ell_{\infty}^{d_l}}$$
is a norm one projection. This, along with
Theorem \ref{thm: la menor injectiva} implies    that
\begin{align*}
(\otimes_{i=1}^l \Pi_i\circ j_i)(B_{\infty}^{d_1}\otimes_{\pi^{inj}}\cdots {\otimes}_{\pi^{inj}} B_{\infty}^{d_l})=&(\otimes_{i=1}^l \Pi_i)(
B_{\ell_{\infty}\hat{\otimes}_{\pi}\cdots \hat{\otimes}_{\pi} \ell_{\infty}}\cap \otimes_{i=1}^l j_i(\mathbb{R}^{d_i}))\\=&
%(\otimes_{i=1}^l j_l)^{-1}\big\{B_{\Pi_{n_1}(\ell_{\infty})\hat{\otimes}_{\pi}\cdots \hat{\otimes}_{\pi} \Pi_{n_l}(\ell_{\infty})}\cap (\otimes_{i=1}^l j_l)(\otimes_{i=1}^{l}\ell^{n_i}_{\infty})\big\}=
B_{\infty}^{d_1}{\otimes}_{\pi}\cdots {\otimes}_{\pi} B_{\infty}^{d_l}.
\end{align*}
Hence, since $\Pi_i\circ j_i$ is the identity on $\mathbb{R}^{d_i},$ we have:
$$B_{\infty}^{d_1}{\otimes}_{\pi^{inj}}\cdots {\otimes}_{\pi^{inj}} B_{\infty}^{d_l}=B_{\infty}^{d_1}{\otimes}_{\pi}\cdots {\otimes}_{\pi} B_{\infty}^{d_l}.$$

The statement concerning the biggest projective tensor product of $B_1^{d_i}$ follows by duality.
\end{proof}
The arguments given for proving Theorem \ref{thm: la menor injectiva} can be adapted to the  case  where the embeddings are not necessarily isometric. The precise  statement is as follows:
\begin{prop}
\label{prop: finite projective injective}
For  $i=1,\ldots l$ and $d_i\in \mathbb{N}$, let $j_i:(\mathbb{E}_{i}, g_{P_i})\hookrightarrow \ell_{\infty}^{d_i}$ be  embeddings with $\|(j_i)\|=1$, $\|(j_i)^{-1}\|\leq M_i$. Then
%$\|j_i\|\,\|(j_i)^{-1}\|\leq M_i$. Then
$$P_{1}\otimes_{\pi^{inj}}\cdots
\otimes_{\pi^{inj}}P_{l}\stackrel{M_1\cdots M_l}{\simeq}(j_1\otimes \cdots\otimes j_l)^{-1}\big\{B_{\infty}^{d_1}{\otimes}_{\pi}\cdots {\otimes}_{\pi} B_{\infty}^{d_l}\cap \otimes_{i=1}^{l}j_i(\mathbb{E}_{i})\big\}.$$
Furthermore,
$$
P_{1}^{\circ}\otimes_{\epsilon^{proj}}\cdots
\otimes_{\epsilon^{proj}}P_{l}^{\circ}\stackrel{M_1\cdots M_l}{\simeq}(j_1^t\otimes \cdots\otimes j_l^t)(B_{1}^{d_1}\otimes_\epsilon\cdots\otimes_\epsilon B_{1}^{d_l}).
$$
\end{prop}
The second part in the last proposition follows  from Corollary \ref{coro: en linfty coinciden}, the uniform property and the first part of the proposition.

\subsection{Relations with some classes of $0$-symmetric convex bodies}\label{subs:John ellipsoid}

Ellipsoids are fundamental tools in the study of both convex bodies and finite dimensional Banach spaces, as can be seen in \cite{MilmanAGA,JohnF,Tomczak-Jaegermann1989}. Below we establish the relation between tensor products of convex bodies and John and L{\"o}wner ellipsoids.
Recall that given a $0$-symmetric convex body $P\subset\mathbb{E},$ its L{\"o}wner ellipsoid $L\ddot{o}w(P)$ and John ellipsoid $John(P)$  are the ellipsoids of minimal (resp. maximal) volume  containing (resp. contained in) $P.$ We suggest \cite[Chapter 3]{Tomczak-Jaegermann1989} for a deeper discussion of this matter.

In the following $\otimes_2$ denotes the Hilbert tensor product of ellipsoids as defined in Section 2.2 or \cite[Section 4]{tensorialbodies}. Since $\otimes_2$ is not defined on the class of $0$-symmetric convex bodies, it is not a tensor product  of $0$-symmetric convex bodies. However, it is not difficult to show that when restricted to the class of ellipsoids, $\otimes_2$ satisfies both properties in Definition \ref{defn: tensor prod of 0symm cb}.

\begin{prop}
\label{prop: relation with lowner john ellipsoids}
 Let $\otimes_\alpha$ be a tensor product of $0$-symmetric convex bodies and let $P_i\subset\mathbb{E}_i,$ $i=1,\ldots,l$ be any tuple of $0$-symmetric convex bodies.
\begin{enumerate}
\item If $\otimes_\alpha\subset\otimes_2$ on the class of ellipsoids, then $L\ddot{o}w(P_1\otimes_\alpha\cdots\otimes_\alpha P_l)=L\ddot{o}w(P_1)\otimes_2\cdots\otimes_2 L\ddot{o}w(P_l).$
\item If $\otimes_2\subset\otimes_\alpha$ on the class of ellipsoids, then $John(P_1\otimes_\alpha\cdots\otimes_\alpha P_l)=John(P_1)\otimes_2\cdots\otimes_2 John(P_l).$
\end{enumerate}
\end{prop}
\begin{proof}
We will prove the result for the L{\"o}wner ellipsoid. The proof for the John ellipsoid follows from the duality between L{\"o}wner and John ellipsoids, and  Proposition \ref{prop:property dual product}.

Suppose that $\otimes_\alpha$ satisfies the inclusion in (1) and
let us denote  $d_i$  the dimension of $\mathbb{E}_i$.
By \cite[Theorem 15.4]{Tomczak-Jaegermann1989}, there exist $m_i$, $x^{i}_{k_i}\in\mathbb{E}_i,$ and positive real numbers $c^{i}_{k_i},$ for $k_i=1,\ldots,m_i,$ s.t. $g_{P_i}(x_{k_i}^i)=g_{P^{\circ}_i}(x_{k_i}^i)=g_{L\ddot{o}w(P_i)}(x_{k_i}^i)=1,$ $I_{\mathbb{E}_i}=\sum_{k_i=1}^{m_i} c^{i}_{k_i}\langle\cdot,x_{k_i}^{i}\rangle_{L\ddot{o}w(P_i)}x^{i}_{k_i}$ and $\sum_{k_i=1}^{m_i} c^{i}_{k}=d_i.$ From this and (1) in Definition \ref{defn: tensor prod of 0symm cb}, it follows that $I_{\otimes_{i=1}^{l}\mathbb{E}_i}$ has a representation as in \cite[Theorem 15.4]{Tomczak-Jaegermann1989} with the vectors $x_{k_1}^{1}\otimes\cdots\otimes x_{k_l}^{l},$ the scalars $c^{1}_{k_1}\cdots c^{l}_{k_l}$ and the scalar product associated to $L\ddot{o}w(P_1)\otimes_2\cdots\otimes_2 L\ddot{o}w(P_l).$  Furthermore, since $\otimes_\alpha\subset\otimes_2,$ then $P_1\otimes_\alpha\cdots\otimes_\alpha P_l\subset L\ddot{o}w(P_1)\otimes_\alpha\cdots\otimes_\alpha L\ddot{o}w(P_l)\subset L\ddot{o}w(P_1)\otimes_2\cdots\otimes_2 L\ddot{o}w(P_l).$ Hence, by the uniqueness of the L{\"o}wner ellipsoid, $L\ddot{o}w(P_1\otimes_\alpha\cdots\otimes_\alpha P_l)=L\ddot{o}w(P_1)\otimes_2\cdots\otimes_2 L\ddot{o}w(P_l).$
\end{proof}

Observe that $(1)$  in  Proposition \ref{prop: relation with lowner john ellipsoids} always holds for $\otimes_\alpha=\otimes_\pi$.  This case was already proved in
 \cite[Lemma 1]{Aubrun2006} .
   It is worth to notice  that also $\otimes_2\subset\otimes_\epsilon$ is always fulfilled. Thus,  Proposition \ref{prop: relation with lowner john ellipsoids} proves that the John ellipsoid of the injective tensor product of $0$-symmetric convex bodies is the Hilbert tensor product of the  John ellipsoids.
\begin{rem}
The inclusions in Proposition \ref{prop: relation with lowner john ellipsoids} can not be dropped. To prove this, observe that $John(B_2^{d_1}\otimes_\pi\cdots\otimes_\pi B_2^{d_l})\subset B_2^{d_1}\otimes_\pi\cdots\otimes_\pi B_2^{d_l}$ is not the Hilbert tensor product of $B_2^{d_i},$ $i=1,\ldots,l.$
\end{rem}

The class of $0$-symmetric convex bodies  with enough symmetries (the unit balls  $B_{p}^{d}$, for example) is a well known class of convex sets used, among others, to study  Banach-Mazur distances, see \cite[$\mathsection$16]{Tomczak-Jaegermann1989}.
Given a $0$-symmetric convex body $P\subset\mathbb{E}$, the set of symmetries of $P$ is denoted by $Iso(P).$ It consists of the linear maps $w:\mathbb{E}\rightarrow\mathbb{E}$ such that $wP=P.$ Recall that a $0$-symmetric convex body $P$ has enough symmetries if the only linear maps $T:\mathbb{E}\rightarrow\mathbb{E}$ that conmute with each $w\in Iso(P)$ are the scalar multiples of the identity map $I_\mathbb{E}$.

\begin{prop}
\label{prop:convex bodies with enough symmetries}
Let $\otimes_\alpha$ be a tensor product of $0$-symmetric convex bodies. If $P_i\subset\mathbb{E}_i$, $i=1,\ldots,l,$ are $0$-symmetric convex bodies with enough symmetries then $P_1\otimes_\alpha\cdots\otimes_\alpha P_l$ has enough symmetries too.
\end{prop}

\begin{proof}
The result is proved using an inductive argument. Note that the case of one factor is trivial.  We briefly expose the proof for the case of three factors. The case of two factors and the general one,  follow analogously. Induction is used to prove that the  corresponding equalities (\ref{eq: induction case order l}) imply  the existence of the scalar $\lambda$ below.
 In \cite[Excercise 4.27]{Aubrun2017}  one finds  the case of two factors for the product $\otimes_\pi$.

 To that end, observe that  by Proposition \ref{prop:linear invariance of products}, $w_1\otimes w_2\otimes w_3\in Iso(P_1\otimes_\alpha P_2\otimes_\alpha P_3)$ for every $w_i\in Iso(P_i).$ Let $S:\mathbb{E}_{1}\otimes\mathbb{E}_{2}\otimes\mathbb{E}_{3}\rightarrow\mathbb{E}_{1}\otimes\mathbb{E}_{2}\otimes\mathbb{E}_{3}$ be a linear map that conmutes with the elements of $Iso(P_1\otimes_\alpha P_2\otimes_\alpha P_3).$

For each $a^2,b^2\in\mathbb{E}_{2}$ and $a^3,b^3\in\mathbb{E}_{3},$ let $S_{a^2b^2a^3b^3}$ be the linear map on $\mathbb{E}_{1}$ defined via duality as
$\langle S_{a^2b^2a^3b^3}(x),x^{\prime}\rangle_{\mathbb{E}_1}=\langle S(x\otimes a^2\otimes a^3),x^{\prime}\otimes b^2\otimes b^3\rangle_H.$ It is not difficult to see that $S_{a^2b^2a^3b^3}$ conmutes with each $w_1\in Iso(P_1).$ Hence, $S_{a^2b^2a^3b^3}=\lambda_{a^2b^2a^3b^3}I_{\mathbb{E}_1}.$ In a similar way, one can define $S_{a^1b^1a^3b^3}:\mathbb{E}_2\rightarrow\mathbb{E}_2,$ $S_{a^1b^1a^2b^2}:\mathbb{E}_3\rightarrow\mathbb{E}_3$ for every $a^1b^1\in\mathbb{E}_1.$ Clearly, $S_{a^1b^1a^3b^3}=\lambda_{a^1b^1a^3b^3}I_{\mathbb{E}_2},$  $S_{a^1b^1a^2b^2}=\lambda_{a^1b^1a^2b^2}I_{\mathbb{E}_3}$ also hold. Thus by evaluating on $a^i,b^i\in\mathbb{E}_i,$ $i=1,2,3,$ we have:
\begin{equation}
\label{eq: induction case order l}
 \lambda_{a^2b^2a^3b^3}\langle a^1,b^1\rangle_{\mathbb{E}_1}=\lambda_{a^1b^1a^3b^3}\langle a^2,b^2\rangle_{\mathbb{E}_2}=\lambda_{a^1b^1a^2b^2}\langle a^3,b^3\rangle_{\mathbb{E}_3},
\end{equation}
for every $a^i,b^i\in\mathbb{E}_i$. The latter implies that, $\lambda_{a^2b^2a^3b^3}=\lambda\langle a^1,b^1\rangle_{\mathbb{E}_1},$ $ \lambda_{a^1b^1a^3b^3}=\lambda\langle a^2,b^2\rangle_{\mathbb{E}_2}$ and $\lambda_{a^1b^1a^2b^2}=\lambda\langle a^3,b^3\rangle_{\mathbb{E}_3}$ for some $\lambda.$ Hence, $S=\lambda I_{\mathbb{E}_1\otimes\mathbb{E}_2\otimes\mathbb{E}_3}.$
\end{proof}

\section{Order 2 tensor products of $0$-symmetric convex bodies and ideals of linear operators}\label{sect: Hilbertian tsr pdct}

In the case of products of order two, the theory   of   tensor products of Banach spaces   is closely related to the theory of operators ideals. This fact  can be found  in  the monographs \cite{defantfloret, diestelfourie, Ryan2013, Tomczak-Jaegermann1989}.

 Following the usual terminology,  a \textbf{Banach operator
ideal} consists of an assignment to each pair of Banach spaces $X,Y$
of a vector space $\mathcal{A}\left(X,Y\right)$
of bounded linear operators from $X$ to $Y,$ together with a norm
$A$ on this space, with the following properties:

\begin{enumerate}
\item $\mathcal{A}\left(X,Y\right)$ is a linear subspace of $\mathcal{L}\left(X,Y\right)$
which contains the finite rank operators. Furthermore, for every $\varphi\in X^{*}$
and $y\in Y$, $A\left(\varphi\left(\cdot\right)y\right)=\left\Vert \varphi\right\Vert \left\Vert y\right\Vert $.

\item The ideal property: if $S\in\mathcal{A}\left(X_{0},Y_{0}\right),T\in\mathcal{L}\left(X,X_{0}\right)$
and $R\in\mathcal{L}\left(Y_{0},Y\right)$ then $RST\in\mathcal{A}\left(X,Y\right)$
and $A\left(RST\right)\leq\left\Vert R\right\Vert A\left(S\right)\left\Vert T\right\Vert .$

\item $\left(\mathcal{A}\left(X,Y\right),A\right)$ is a Banach space.
\end{enumerate}
If we only consider finite dimensional normed spaces, then $\mathcal{A}(M,N)=\mathcal{L}(M,N)$ with the norm $A$. Thus, in this context, we say that $A$ is an \textbf{ideal norm}. For a complete treatment of Banach operator ideals see \cite{defantfloret, pietsch1980operator}, and for the case of finite dimensions see \cite{Tomczak-Jaegermann1989}.

Let $u=\sum_{i=1}^{n}x_{i}\otimes y_{i}$ be an element of the tensor
product $M\otimes N$ of finite dimensional normed spaces. By $T_{u}$
we denote the linear map:
\begin{alignat*}{1}
T_{u}:M^{*} & \rightarrow N\\
x^{*} & \rightarrow\sum_{i=1}^{n}x^{*}\left(x_{i}\right)y_{i}.
\end{alignat*}
Conversely, let $T=\sum_{i=1}^{n}x_{i}^{*}\left(\cdot\right)y_{i}$
be a linear map from $M$ to $N.$ By $u_{T}$ we denote the vector $\sum_{i=1}^{n}x_{i}^{*}\otimes y_{i}$ that
belongs to $M^{*}\otimes N$.
The mappings $u_T\mapsto T_u$ and $T_u\mapsto u_T$ establish a natural bijection between tensor norms on finite dimensional spaces and ideal norms (see 17.1 and 17.2 of \cite{defantfloret}). Then, along with Theorem \ref{thm: main theorem the bijection}, we have:

\begin{cor}
\label{cor: ideals tensor norms tensor convex}
Given  a tensor product of order $2$ of $0$-symmetric
convex bodies $\otimes_{\alpha}$, for every pair of finite dimensional normed
spaces  $M,N$, define:
\[
A\left(T:M\rightarrow N\right):=\left\Vert u_{T}\right\Vert _{\otimes_{\alpha}}.
\]
Then $A$ is an ideal norm.
\end{cor}

Indeed, we have that the mappings in the following diagram, which are as in Theorem \ref{thm: main theorem the bijection} and Corollary \ref{cor: ideals tensor norms tensor convex}, are bijections that respect duality, injectivity and projectivity:

     \begin{equation}\label{eq: ternaria}
\begin{tikzcd}[cramped, sep=small]
      & \left\{{\begin{tabular}{c}\text{Tensor products of}   \\ $0$-symm. convex bodies \\ {$\otimes_{\alpha}$}
   \end{tabular} }\right\} \arrow[leftrightarrow,shift left=-.5cm]{rd}\arrow[leftrightarrow, shift left=.5cm]{ld} & \\
  \left\{{\begin{tabular}{c}\text{Tensor norms on}\\ f.d. normed spaces \\ {$\|\cdot\|_{\alpha}$}
   \end{tabular} }\right\} \arrow[leftrightarrow]{rr}     & & \left\{{\begin{tabular}{c}\text{Ideals of }\\ operators on f.d.  \\  normed spaces \\ $\mathcal{A}$
   \end{tabular}}\right\}. \\
 \end{tikzcd}
\end{equation}

% \arrow[leftrightarrow]{rd}\arrow[leftrightarrow,shift left=.5cm]{ld}
%On finite dimensions, spaces with enough symmetries have many desirable properties. By way of example, it is well known that on this spaces the Banach-Mazur distance to the Hilbert space can be calculated using the formal identity map and the LOWNER ellipsoid REFERENCE

\subsection{The Hilbertian tensor product of $0$-symmetric convex bodies}\label{subsect: Hilbertian tsr pdct}

Let \-$(\mathcal{L}_{H}, \|\cdot\|_{H})$ be the operator ideal of bounded operators which factor through a Hilbert space. If $T\in \mathcal{L}_{H}(X; Y)$, then $\|T\|_{H}=inf\{\|R\|\,\|S\|\}$ where the infimum runs over all possible factorizations. Here we follow the notation in  \cite{diestelfourie} and \cite{Ryan2013}.  In \cite{PisierC60}
 this ideal and the norm  are denoted $(\Gamma_2,\gamma_2)$. The tensor norm  $\omega_2$ associated to this ideal  is known as the Hilbertian tensor norm.  It      is profusely  studied in \cite{defantfloret}, \cite{diestelfourie}, \cite{Ryan2013}. It is an injective tensor norm.
  A fundamental result in the theory of tensor norms involves $\omega_2,$ it is the so called  {\sl Grothendieck`s Theorem}. It is formulated, in terms of tensor norms, as:
\begin{equation}
\label{eq: Grothendieck ineq}
\omega_2(u)\leq\pi(u)\leq K_{G}\omega_2(u) \text{ for } u\in\ell_\infty^m\otimes\ell_\infty^n,
\end{equation}
% Among its more relevant propierties, we find that it is equivalent to the smaller injective tensor norm, and that it is equivalent to the projective tensor norm on spaces of type $C(K)$. These facts are a consequence of the so called Grothendieck`s Theorem (see \cite{Ryan, Pisier, DefantFloret, DiestelFouri}).
%%Ryan seccion 7.4 pag 176. Pisier cap 2 pag 21. DefantFloret pag 153 169., DiestelFouri Cap 3
where $K_G$ is Grothendieck's constant and $n, m\in\mathbb{N}.$ This form of Grothendieck's theorem appears as \cite[14.4]{defantfloret}. We suggest the reader the monograph \cite{pisier} on the subject.

  According to (\ref{eq: ternaria}), let us define $\otimes_{\omega_2}$ as the Hilbertian  tensor product of $0$-symmetric convex bodies associated to the tensor norm  $\omega_2$ and the ideal $\|\cdot\|_{H}$. Thus, by means of Theorem \ref{thm: main theorem the bijection}, Grothendieck`s theorem can be written in terms of $0$-symmetric convex bodies as:
\begin{equation}
\label{eq: groth ineq elinfty}
   {K_G}^{-1}   B_{\infty}^m\otimes_{\omega_2}B_{\infty}^n \subset B_{\infty}^m\otimes_{\pi}
     B_{\infty}^n\subset B_{\infty}^m\otimes_{\omega_2}B_{\infty}^n.
   \end{equation}
%  \begin{proof}  By Theorem \ref{thm: main theorem the bijection},  $\omega_2$ is a tensor product of $0$-symmetric convex bodies. So, it must satisfy condition $(1)$ in Definition \ref{defn: tensor prod of 0symm cb}. This proves  the second inclusion in (\ref{eq: groth ineq elinfty}).
%  To prove the first inclusion, observe    that by  Proposition \ref{prop:tensor norms determine tensor sym cb}, it holds
%  $$B_{\ell_{\infty}^n}\otimes_{\omega_2}B_{\ell_{\infty}^m}=
%  B_{\ell_{\infty}^n\otimes_{w_2} \ell_{\infty}^m}\quad \mbox{and}\quad  B_{\ell_{\infty}^n}\otimes_{\pi} B_{\ell_{\infty}^m}=B_{\ell_{\infty}^n\otimes_{\pi} \ell_{\infty}^m}.$$
%  The proof follows now by using a predual form of Grothendieck`s Theorem, as in Theorem \cite[Theorem 3.1]{Pisier11}, which states:
% $
%  B_{\ell_{\infty}^n\otimes_{w_2} \ell_{\infty}^m} \subset K_G B_{\ell_{\infty}^n\otimes_{\pi} \ell_{\infty}^m}. $
%   \end{proof}
By Theorem \ref{thm: main theorem the bijection}, $\otimes_{\omega_2}$   is an injective tensor product. Combining these results, it also holds:

   \begin{prop}\label{prop: Groth Thm mayor inyectiva}Let $P_i\subset\mathbb{E}_i$, $i=1,2$ be $0$-symmetric convex bodies. Then
    \begin{equation}\label{eq: groth ineq elinfty}
   {K_G}^{-1}   P_1\otimes_{\omega_2}P_2 \subset P_1\otimes_{{\pi^{inj}}}
    P_2\subset P_1\otimes_{\omega_2}P_2.
   \end{equation}
   \end{prop}
\begin{proof} The result in terms of tensor norms can be found in  \cite[Theorem 7.29]{Ryan2013}. Now, it is enough to  apply  the transition   to the $0$-symmetric convex bodies setting,  that
 Theorem \ref{thm: main theorem the bijection} and  Proposition \ref{prop: injectiva y projectiva a partir del teorema principal} provide.
\end{proof}

     These results give rise to  an explicit representation (up to a  constant) of the tensor product $\otimes_{\omega_2}$. In the following corollary it is worth noticing that, by   Dvoretsky`s Theorem (\cite[Theorem 5.1.2]{MilmanAGA}), one may  consider $M_i\leq 1+\epsilon$ for an arbitrary $\epsilon >0$, at the price of growing the dimensions $d_i$:

  \begin{cor}\label{coro: Groth Thm}
  Let $P_i\subset\mathbb{E}_i$, $i=1,2,$ be $0$-symmetric convex bodies and let  $j_i:(\mathbb{E}_{i}, g_{P_i})\hookrightarrow \ell_{\infty}^{d_i}$ be  embeddings with $\|j_i\|=1, \,\|(j_i)^{-1}\|\leq M_i$, as in Proposition \ref{prop: finite projective injective}. Then
 \begin{align*}P_{1}\otimes_{\omega_2}
P_{2}\subset K_G  (j_1& \otimes j_2)^{-1} \big\{B_{\infty}^{d_1}\otimes_{\pi} B_{\infty}^{d_2}\cap (j_1\otimes j_2)(\mathbb{E}_{1}\otimes \mathbb{E}_{2})\big\}  \\ & \subset K_GM_1 M_2
P_{1}\otimes_{\omega_2}P_2.
\end{align*}
In the case where $P_1:=\mathcal{E}_1$ is an  ellipsoid (the case  $P_2:=\mathcal{E}_2$ is analogue)
\[
{\mathcal{E}}_1\otimes_{\epsilon}P_2=
{\mathcal{E}}_1\otimes_{\omega_2}P_2
\stackrel{K_GM_1 M_2}{\simeq}
 (j_1\otimes j_2)^{-1} \big\{B_{\infty}^{d_1}\otimes_{\pi} B_{\infty}^{d_2}\cap (j_1\otimes j_2)(\mathbb{E}_{1}\otimes \mathbb{E}_{2})\big\}  .
\]
  \end{cor}
\begin{proof} The first part follows from Proposition \ref{prop: Groth Thm mayor inyectiva}
 and Proposition \ref{prop: finite projective injective}.
 To prove the assertion on the tensor product of ellipsoids, we are using the notation  ${\simeq}$  to mean the norm of the   isomorphisms, which is  as above.

 To prove the assertion on  ellipsoids, let us denote $F:=(\mathbb{E}_2,g_{P_2})$. Observe that $(\mathcal{L}(H_1^*,F),\|\cdot\|)$ is isometric to
 $(\mathcal{L}(H_1^*,F),\|\cdot\|_H)$ since, trivially,  every bounded operator factorizes through a Hilbert space. By Theorem \ref{thm: main theorem the bijection} and (\ref{eq: ternaria}), this means that
 ${\mathcal{E}}_1\otimes_{\epsilon}P_2=
 {\mathcal{E}}_1\otimes_{\omega_2}P_2$.
 The isomorphic expression as a section of the projective tensor product of $\ell_{\infty}$-spaces follows from the previous one.
 \end{proof}

The dual tensor product $\otimes_{\omega_2^{\prime}}$ is well defined, according to Definition \ref{defn: dual tensor prod of cb}. Also, it is a projective tensor product. By Theorem \ref{thm: main theorem the bijection}, it corresponds to the dual tensor norm $\omega_2^{\prime}.$
Even more, from Proposition \ref{prop: Groth Thm mayor inyectiva}, it follows that
$\otimes_{\omega_2^{\prime}}$   is equivalent to the biggest projective tensor product ${\otimes_{\epsilon^{proj} }}$. That is:
%and Proposition \ref{prop: Groth Thm mayor inyectiva}, $\otimes_{\omega_2^{\prime}}$   is equivalent to the biggest projective tensor product ${\otimes_{\epsilon / }}$. That is:
\begin{cor}
\label{eq:dual Hibertian product}
For every pair of $0$-symmetric convex bodies $P_i\subset\mathbb{E}_i,$ $i=1,2,$ the following holds:
\[
P_1\otimes_{\omega^{\prime}_2}P_2 \subset P_1\otimes_{{\epsilon^{proj}}}
    P_2\subset {K_G}   P_1\otimes_{\omega^{\prime}_2}P_2.
\]
\end{cor}

As a consequence of the last corollary and Proposition \ref{prop: finite projective injective}, both of the statements in Corollary \ref{coro: Groth Thm}  have dual expresions for the tensor products $\otimes_{\omega_2^{\prime}}$ and $\otimes_{\epsilon}.$

Let  $P_i\subset\mathbb{E}_i$ be $0$-symmetric convex bodies and $q_i:\ell_{1}^{d_i}\hookrightarrow(\mathbb{E}_{i}, g_{P_i}),$ $i=1,2,$ be surjective linear maps such that $q_{i}(B_1^{d_i})\subset P_i\subset M_iq_{i}(B_1^{d_i})$, for $M_i>0,$ $i=1,2$. Then, 
\begin{equation}
\label{eq:dual Groth Thm 1}
\frac{1}{M_1M_2}P_{1}\otimes_{\omega_2^{\prime}}
P_{2}\subset  (q_1 \otimes q_2) (B_{1}^{d_1}\otimes_{\epsilon} B_{1}^{d_2})\subset K_G
P_{1}\otimes_{\omega_2^{\prime}}P_{2}.
\end{equation}
In the case where $P_1:=\mathcal{E}_1$ is an ellipsoid, it holds:
\begin{equation}
\label{eq:dual Groth Thm 2}
{\mathcal{E}}_1\otimes_{\pi}P_2=
{\mathcal{E}}_1\otimes_{\omega_2^{\prime}}P_2
\stackrel{K_GM_1 M_2}{\simeq}
 (q_1\otimes q_2) (B_{1}^{d_1}\otimes_{\epsilon} B_{1}^{d_2}).
\end{equation}

It is worth to notice that, in the context of tensor norms, Corollary \ref{eq:dual Hibertian product} corresponds to \cite[Corollary 7.30]{Ryan2013}. For a detail treatment of the norm $\omega_2^{\prime},$ we suggest the reader \cite[Section 7.4]{Ryan2013}. 

%Chapter 3 \cite{DiestelFourie} lo denota $\otimes_H$; Pisier Chapter 2 \cite{Pisier CBMS60}  lo denota $\gamma_2$;
%Ryan Section 7.4 \cite{ryan} lo denota $w_2$.
%
%\bibliographystyle{plain}
%\bibliography{bibliopaper2}
%\nocite{*}

\end{document}